\documentclass[ejs]{imsart}

 \RequirePackage[OT1]{fontenc}
\RequirePackage{amsthm,amsmath}
\RequirePackage[colorlinks,citecolor=blue,urlcolor=blue]{hyperref}
\RequirePackage{hypernat}

\usepackage[english]{babel}
\usepackage[normalem]{ulem}
\usepackage{graphicx}	
\usepackage{amsfonts}
\usepackage{amssymb}    
\usepackage{float}
\usepackage{color}

\def\var{\mbox{Var}}

\usepackage{mathrsfs}

\usepackage{geometry}
\newfloat{Figure}{H}{lof}
\newfloat{Table}{Hptb}{lot}

\newenvironment{enumerate*}%
  {\begin{enumerate}%
    \setlength{\itemsep}{0pt}%
    \setlength{\parskip}{0pt}}%
  {\end{enumerate}}

\newtheorem{thrm}{Theorem}[section]
\newtheorem{prte}[thrm]{Proposition}
\newtheorem{lemma}[thrm]{Lemma}
\newtheorem{cor}[thrm]{Corollary}
\newtheorem{defi}[thrm]{Definition}
\newtheorem{remark}{Remark}[section]

\newtheorem{algo}{Algorithm}[section]
\theoremstyle{plain}

\begin{document}

\begin{frontmatter}
\title{Adaptive estimation of covariance matrices via Cholesky decomposition}
\runtitle{Covariance estimation} 

\begin{aug}
\author{\fnms{Nicolas}
  \snm{Verzelen}\thanksref{t1}\ead[label=e1]{nicolas.verzelen@supagro.inra.fr}}
 
\thankstext{t1}{Research mostly carried out at Univ Paris-Sud (Laboratoire de
Mat\'ematiques, CNRS-UMR 8628)}
\runauthor{N. Verzelen} 
 
\affiliation{INRA and SUPAGRO} 
 
\address{
 INRA, UMR  729 MISTEA,\\ 
F-34060 Montpellier, France
} 

\address{
 SUPAGRO, UMR  729 MISTEA,\\ 
F-34060 Montpellier, France\\
\printead{e1}
} 
\end{aug}

\begin{abstract} 
This paper studies the estimation of a large covariance matrix. We introduce a
novel procedure called ChoSelect based on the Cholesky factor of the inverse
covariance. This method uses a dimension reduction strategy by selecting  the
pattern of zero of the Cholesky factor. Alternatively, 
ChoSelect can be interpreted as a graph estimation procedure for directed
Gaussian graphical models. 
Our approach is particularly relevant when the variables under study have a
natural ordering (e.g. time series) or more generally when the Cholesky factor
is approximately sparse.
ChoSelect achieves  non-asymptotic oracle inequalities with respect to the
Kullback-Leibler entropy. Moreover, it satisfies various adaptive properties
from a minimax point of view. We also introduce and study a two-stage procedure
that combines ChoSelect with the Lasso. This last method enables the
practitioner to choose his own trade-off between statistical efficiency and
computational complexity. Moreover, it is consistent under
weaker assumptions than the Lasso.
 The practical performances of the different procedures
are assessed on numerical examples.
\end{abstract}
 

\end{frontmatter}

\maketitle

\section{Introduction}\label{section_introduction}

The problem of estimating large covariance matrices has recently attracted a lot
of attention. On the one hand, there is an inflation of high-dimensional data in
many scientific areas: gene arrays, functional magnetic resonance imaging
(fMRI), image classification, and climate studies. On the other hand, many data
analysis tools require an 
estimation of the covariance matrix $\Sigma$. This is for instance the case  for
principal component analysis (PCA), for linear discriminant analysis (LDA), or
for establishing independences or conditional independences between the
variables. It is known for a long time that the simplest estimator, the sample
covariance matrix performs poorly when the size of the vector $p$ is larger than
the number of observations $n$ (see for instance
Johnstone~\cite{johnstone2001}).

Depending on the objectives of the analysis and on the applications, different
approaches are used for estimating high-dimensional covariance matrices. Indeed,
if one wants to 
perform PCA or to establish independences between the covariates, then it is
advised to estimate directly the covariance matrix $\Sigma$. In contrast,
performing LDA  further relies on the inverse of the covariance matrix. In the
sequel, we call this matrix the precision matrix and note it $\Omega$. Sparse
precision matrices are also of interest because of their connection with
graphical models and conditional independence. The pattern of zero in $\Omega$
indeed corresponds to the graph structure of the distribution (see for instance
Lauritzen~\cite{lauritzen96} Sect.5.1.3). 

Most of the methods based on direct covariance matrix estimation amount to
regularize the empirical covariance matrix.  Let us mention the work of 
Ledoit and Wolf~\cite{ledoit} who propose to replace the sample covariance with
its linear combination with the identity matrix. However, these shrinkage
methods are known to provide an inconsistent estimation of the
eigenvectors~\cite{johnstonelu}. Applying recent results on random matrix
theory, El Karoui~\cite{karoui07} and Bickel and Levina~\cite{bickel08b} have
studied thresholding estimators of $\Sigma$. The resulting estimator is sparse
and is proved (for instance~\cite{bickel08b}) to be consistent with respect to
the operator norm under mild conditions as long as $\log(p)/n$ goes to $0$.
These results are particularly of interest for performing PCA since they imply a
consistent estimation of the eigenvalues and the eigenvectors. Observe that all
these methods are invariant under permutation of the variables. Yet, in many
applications (for instance times series, spectroscopy, climate data), there
exists a natural ordering in the data. In such a case, one should use other
procedures to obtain faster rates of convergence.
Among other,  Furrer and Bentgsson~\cite{furrer} and Bickel and
Levina~\cite{bickel08} use banded or tapering estimators. Again, the consistency
of such estimators is proved.  Moreover, all these methods share an
attractive computational cost. We refer to the introduction of~\cite{bickel08b}
for a more complete review.

 The estimation procedures of the precision matrix $\Omega$ fall into three
categories depending whether there exists an ordering on the variables and to
what extent this ordering is important. If there is not such an ordering,
d'Aspremont et al.~\cite{aspremont}  and Yuan and Lin~\cite{yuan07} have
adopted a penalized likelihood approach by applying a $l_1$ penalty to the
entries  of the precision matrix. It has also been discussed by Rothman et
al.~\cite{rothman07} and Friedman et al.~\cite{friedman08} and extended
by Lam and Fan et al.~\cite{fan_covariance} or Fan et
al.~\cite{Fan08} to other penalization methods. These estimators are known to
converge with respect to the Frobenius norm (for instance~\cite{rothman07}) when
the underlying precision matrix is sparse enough. 

When there is a natural ordering on the covariates, the regularization is
introduced via the Cholesky decomposition:
$$ \Omega=T^* S^{-1}T\ ,$$
 where $T$ is a lower triangular matrix with a unit diagonal and $S$ is a
diagonal matrix with positive entries. The elements of the $i$-th row can be
interpreted as regression coefficient of $i$-th component given its
predecessors. This will be further explained in Section \ref{section_modele}.
For time series or spectroscopy data, it is more likely that the relevant
covariates for this regression of the $i$-th component are its closest
predecessors. In other word, it is expected that the matrix $T$ is approximately
banded.
 With this in mind,  Wu and Pourahmadi~\cite{wu03} introduce a $k$-banded
estimator of the matrix $T$ by smoothing along the first $k$ subdiagonals and
setting the rest to 0. The choice of $k$ is made by applying AIC
(Akaike~\cite{akaike73}). They prove element-wise consistency of their estimator
but did not provide any high-dimensional result with respect to a loss function
such as Kullback or Frobenius. Bickel and Levina~\cite{bickel08} also consider
$k$-banded estimator of $T$ and are able to prove rates of convergence in the
matrix  operator norm. Moreover, they introduce a cross-validation approach for
choosing a suitable $k$, but they do not prove that the selection method
achieves adaptiveness. More recently, Levina et al.~\cite{levina08}
propose a new banding procedure based on a nested Lasso penalty. Unlike the
previous methods, they allow the number $k=k_i$ used for banding to depend on
the line $i$ of $T$. They do not state any theoretical result, but they exhibit
numerical evidence of its efficiency. In the sequel, we call the issue of
estimating $\Omega$ by banding the matrix $T$ the \emph{banding problem}.

Between the first approach based on precision matrix regularization and the
second one which relies on banding the Cholesky factor, there exists a third one
which is not permutation invariant, but does not assume that the matrix $T$ is
approximately banded. It  consists in approximating $T$ by a sparse lower
triangular matrix (i.e. most of the entries are set to $0$).

When is it interesting to adopt this approach? If we consider
a directed graphical model whose graph is sparse and  compatible with the
ordering of the variables, then the Cholesky factor $T$ is sparse. Indeed, its
pattern of zero is related to the directed acyclic graph (DAG) of the directed
graphical model associated to this ordering (see Section \ref{section_modele}
for a definition).  More generally, it may be worth using this strategy even if
one does not know a ``good'' ordering on the variables. 
 On the one hand, most of the procedures based on the estimation of $T$ are
computationally faster than their counterpart based on the estimation of
$\Omega$. This is due to the
decomposition of the likelihood into $p$ independent terms explained in Section
\ref{section_description}. On the other hand, there exist examples
of sparse Cholesky factor $T$ such that the precision matrix $\Omega$ is not
sparse at all. Consider for instance a matrix $T$ which is zero except on the
diagonal and on the last line. Admittedly, it
is not completely satisfying to apply a method that depends on the ordering of
the variables when we do not know a \emph{good} ordering. There are indeed
examples of sparse precision matrices $\Omega$ such that for a \emph{bad}
ordering, the Cholesky factor is not sparse at all (see \cite{rothman07}
Sect.4).  Nevertheless, if sparse precision matrices and sparse Cholesky factors
have different approximation capacities, it remains still unclear which one
should be favored.

In the sequel, we call the issue of estimating $T$ in the class of sparse lower
triangular matrices the \emph{complete graph selection}  problem by analogy to
the complete variable estimation problem in regression problems. In this
setting, Huang et al.~\cite{huang06} propose to  add an $l_1$ penalty on
the elements of $T$. More recently, Lam and Fan~\cite{fan_covariance} have
extended the method to other types of penalty and have proved its consistency in
the Frobenius norm if the matrix $T$ is exactly sparse. To finish, let us
mention that Wagaman and Levina~\cite{wagaman08} have developed a data-driven
method based on the isomap algorithm for picking a ``good'' ordering on the
variables. \\

In this paper, we consider both the banding problem and the complete graph
selection problem.
We introduce a general $l_0$ penalization method based on maximum likelihood for
estimating the matrices $T$ and $S$. 
We exhibit a  non-asymptotic oracle inequality with respect to the Kullback loss
\emph{without} any assumption on the target $\Omega$. 

For the adaptive banding issue, our method is shown to achieve the optimal rate
of convergence and is adaptive to the rate of decay of the entries of $T$ when
one moves away from the diagonal. Corresponding minimax lower bounds are also
provided. We
also compute asymptotic rates of convergence in the Frobenius norm. Contrary to
the $l_1$ penalization methods, we explicitly provide  the constant for tuning
the penalty. Finally, the method is  computationally efficient.

For complete graph selection, we prove that our estimator
non-asymptotically achieves the optimal rates of convergence  when $T$ is
sparse. We also provide the corresponding minimax lower bounds. To our
knowledge, this minimax lower bounds with respect to the Kullback discrepansy
are also new.  Moreover, our method is flexible and allows to integrate some
prior knowledge on the graph. However, this procedure is computationally
intensive which makes it infeasible for $p$ larger than $30$. This is why we
introduce in Section \ref{section_fast_algorith} a computationally faster
version of the estimator by applying a two-stage procedure. This method inherits
some of the good properties
of the previous method and applies for arbitrarily large $p$. Moreover, it is
shown to select consistently the pattern of zeros under weaker assumptions than
the Lasso. These theoretical
results are corroborated by a simulation study. \\

Since data analysis methods like LDA are based on likelihood we find more
relevant to obtain rates of convergence with respect to the Kullback-Leibler
loss than Frobenius rates of convergence. Moreover, considering Kullback loss
allows us to obtain rates of convergence which are free of hidden dependency on
parameter such as the largest eigenvalue of $\Sigma$. In this sense, we argue
that this loss function is more natural for the statistical problem under
consideration. \\

The paper is organized as follows. Section \ref{section_preliminaire} gathers
some preliminaries about the Cholesky decomposition and introduces the main
notations. 
In Section \ref{section_description}, we describe the procedure and provide an
algorithm for computing the estimator $\widetilde{\Omega}$.  In Section
\ref{section_theorique},  we state the main result of the paper, namely a
general non-asymptotic oracle type inequality for the risk of
$\widetilde{\Omega}$. In Section \ref{section_ordered}, we specify our result to
the problem of adaptive banding. Moreover, we prove that our so-defined
estimator is minimax adaptive to the decay of the off-diagonal coefficients of
the matrix $T$. Asymptotic rates of convergence with respect to the Frobenius
norm are also provided. In Section \ref{section_complete}, we investigate the
complete graph selection issue. We first derive a non-asymptotic oracle
inequality and then derive that our procedure is minimax adaptive to the unknown
sparsity of the Cholesky factor $T$. As previously, we provide asymptotic rates
of convergence with respect to the Frobenius loss function. Moreover, we
introduce a computationally feasible estimation procedure in Section
\ref{section_fast_algorith} and we derive an oracle-type inequality and
sufficient condition for consistent selection of the graph. In Section
\ref{section_simulation}, the performances of the procedure are assessed on
numerical examples for both the banding and the complete graph selection
problem. We make a few concluding remarks in Section \ref{section_discussion}.
Sketch of the proof are in Section \ref{section_proofs}, while the details are
postponed to the technical Appendix~\cite{technical}.

\section{Preliminaries}\label{section_preliminaire}

\subsection{Link with conditional regression and graphical
models}\label{section_modele}

In this subsection, we review basic properties about Cholesky factors and
explain their connection with directed graphical models. \\

We consider the estimation of the vector $X=(X_i)_{1\leq i\leq  p}$ of size $p$
which follows a centered normal distribution with covariance matrix $\Sigma$. We
always assume that $\Sigma$ is non-singular. We recall that the precision matrix
$\Omega$ uniquely decomposes as $\Omega=T^*ST$ where $T$ is a lower triangular
matrix with unit diagonal and $S$ is a diagonal matrix.
Let us first emphasize the connection between the modified Cholesky factor $T$
and conditional regressions. 
For any $i$ between $2$ and $p$ we note $t_{i}$ the vector of size $i-1$ made of
the $i-1$-th first elements of the $i$th-line of $T$. By convention $t_1$ is the
vector of null size. Besides, we note $s_i$ the $i$-th diagonal element of the
matrix $S$. 
Let us define the vector $\epsilon=(\epsilon_i)_{1\leq i\leq p}$ of size $p$ as
$\epsilon:=T X$. By standard Gaussian properties, the covariance matrix of
$\epsilon$ is $S$. Since the diagonal of $T$ is one, it follows that for any
$1\leq i\leq p$
\begin{eqnarray}\label{modele_regression}
 X[i] = \sum_{j=1}^{i-1}-t_{i}[j]X[j]+\epsilon_i \ ,
\end{eqnarray}
where $\var(\epsilon_i)=s_i$ and the $(\epsilon_i)_{1\leq i\leq p}$ are
independent. \\

Let $\overrightarrow{G}$ be a directed acyclic graph who vertex set is
$\{1,\ldots,p\}$. We assume that the direction of the edges is compatible with
the natural ordering of $\{1,\ldots,p\}$. In other words, we assume that any
edge $j\rightarrow i$ in $\overrightarrow{G}$ satisfies $j<i$. Given a vertex
$i$, the set of its parents is defined by:
$$pa_{\overrightarrow{G}}(i):=\left\{j<i\, ,\, j\rightarrow i \right\}\ .$$
Then, the vector $X$ is said to be a {\bf directed Gaussian graphical model}
with respect to $\overrightarrow{G}$ if for any $1\leq j<i \leq p$ such that
$j\notin pa_{\overrightarrow{G}}(i)$, $X_i$ is independent of $X_j$
conditionally to $(X_k)_{k\in pa_{\overrightarrow{G}}(i)}$. This means that only
the variables $(X_k)_{k\in pa_{\overrightarrow{G}}(i)}$ are relevant for
predicting $X_i$ among the variables $(X_k)_{k<i}$. There are several
definitions of directed Gaussian graphical model (see Lauritzen
\cite{lauritzen96}), which are all equivalent when $\Sigma$ is non-singular. 

There exists a correspondence between the graph $\overrightarrow{G}$  and the
Cholesky factor $T$ of the precision matrix $\Omega$. If $X$ is a directed
graphical model with respect to $\overrightarrow{G}$, then $T[i,j]=0$ for any
$j<i$ such that $j\nrightarrow i$. Conversely, $X$ is a directed graphical model
with respect to the graph $\overrightarrow{G}$ defined by $j\rightarrow i$ if
and only $T[i,j]\neq 0$. Hence, it is equivalent to estimate the pattern of zero
of $T$ and the minimal graph $\overrightarrow{G}$ compatible with the ordering. 

These definitions and properties depend on a particular ordering of the
variables. It is
beyond the scope of this paper to discuss the graph estimation when the ordering
is not fixed. We refer the interested reader to Kalisch and B\"uhlmann
\cite{buhlmann07}.

\subsection{Notations}

For any set $A$, $|A|$ stands for its cardinality. We are given $n$ independent
observations of the  random vector $X$. We always assume that $X$ follows a
centered Gaussian distribution $\mathcal{N}(0_p,\Sigma)$. In the
sequel, we
note ${\bf X}$ the $n\times p$ matrix of the observations. Moreover, for any
$1\leq i\leq p$ and any subset $A$ of $\{1,\ldots,p-1\}$, ${\bf X}_i$ and ${\bf
X}_A$ respectively refer to the vector of the $n$ observations of $X_i$ and to
the $n\times |A|$ matrix of the observations of $(X_i)_{i\in A}$.

In the sequel, $\mathcal{K}(\Omega;\Omega')$ stands for the Kullback divergence
between the centered normal distribution with covariance $\Omega^{-1}$ and the
centered normal distribution with covariance $\Omega'^{-1}$.  We shall also
sometimes assess the performance of the procedures using the Frobenius norm and
the $l_2$ operator norm. This is why we respectively define
$\|A\|^2_F:=\sum_{i,j}A[i,j]^2$ and $\|A\|$ as the Frobenius norm and the $l_2$
operator norm of the matrix $A$. For any matrix $\Omega$,
$\varphi_{\text{\text{max}}}(\Omega)$ stands for the largest eigenvalue of
$\Omega$. Finally, $L$, $L_1$, $L_2$,$\ldots$ denote universal constants that
can vary from line to line. The notation $L_{.}$ specifies the dependency on
some
quantities.

\section{Description of the procedure}\label{section_description}

In this section, we introduce our procedure for estimating $\Omega$ given a
$n$-sample of the vector $X$.
For any $i$ between $1$ and $p$, $m_i$ stands for a subset of $\{1,\ldots,
i-1\}$. By convention, $m_1=\emptyset$. In terms of directed graphs, $m_i$
stands for the set of parents of $i$. Besides, we call any set $m$ of the form
$m=m_1\times m_2\times \ldots \times m_p$ a model. This model $m$ is one to one
with a directed graph whose ordering is compatible with the natural ordering of
$\{1,\ldots,p\}$. We shall sometimes call $m$ a graph in order to emphasize the
connection with graphical models.

Given a model  $m$, we define $\mathcal{T}_m$ as the affine space of lower
triangular matrices $T$ with unit diagonal such for any $i$ between
$1$ and $p$, the support (i.e. the non-zero coefficients) of $t_i$ is included
in $m_i$. We note $Diag(p)$ the
set of all diagonal matrices with positive entries on the diagonal. The matrices
$\widehat{T}_m$ and $\widehat{S}_m$ are then defined as the maximum likelihood
estimators of $T$ and $S$
\begin{eqnarray}\label{definition_TS}
 \left(\widehat{T}_m,\widehat{S}_m\right)= \arg\min_{T'\in \mathcal{T}_m,\ S'\in
\text{Diag}(p)} \mathcal{L}_n(T,S):=\frac{1}{2}tr\left[T^*S^{-1}T\overline{{\bf
X}^* {\bf X}}\right]+\frac{1}{2}\log|S|
\end{eqnarray}
Here, $\mathcal{L}_n(T,S)$ stands for the negative $\log$-likelihood.
Hence, the estimated precision matrix is $\widehat{\Omega}_m=
\widehat{T}_m^*\widehat{S}_m^{-1}\widehat{T}_m$. This matrix
$\widehat{\Omega}_m$ is the maximum likelihood estimator of $\Omega$ among the
precision matrices which correspond to directed graphical models with respect to
the graph $m$.

For any $i$ between $1$ and $p$, $\mathcal{M}_i$ refers to a collection of
subsets of $\{1,\ldots,i-1\}$ and we call $\mathcal{M}:=\mathcal{M}_1\times
\ldots \times \mathcal{M}_p$ a collection of models (or graphs). The choice of
the collection $\mathcal{M}$ depends on the estimation problem under
consideration. For instance, we shall use a collection corresponding to banded
matrices when we will consider the banding problems. The collections
$\mathcal{M}$ are specified for the banding problem and the complete graph
selection problem in Sections \ref{section_ordered} and \ref{section_complete}.

Our objective is to select a model $\widehat{m}\in\mathcal{M}$ such that the
Kullback-Leibler risk $\mathbb{E}[\mathcal{K}(\Omega;\widehat{\Omega}_m)]$ is as
small as possible. We achieve it through penalization. For any $1\leq i\leq p$,
$pen_i:\mathcal{M}_i\rightarrow \mathbb{R}^+$ is a positive function that we
shall explicitly define later. The penalty function $pen:\mathcal{M}\rightarrow
\mathbb{R}^+$ is defined as $pen(m)=\sum_{i=1}^p pen_i(m_i)$. Then, we select a
model $\widehat{m}$ that minimizes the following criterion
\begin{eqnarray*}
 \widehat{m}:=\arg\min_{m\in\mathcal{M}}2\mathcal{L}_n(\widehat{T}_m,\widehat{S}
_m)+pen(m)= \arg\min_{m\in\mathcal{M}}tr\left[\widehat{\Omega}_m\overline{{\bf
X}^* {\bf X}}\right]-\log|\widehat{\Omega}_m| +pen(m)
\end{eqnarray*}
For short, we write $\widetilde{\Omega}:=\widehat{\Omega}_{\widehat{m}}$,
$\widetilde{S}:=\widehat{S}_{\widehat{m}}$, and
$\widetilde{T}=\widehat{T}_{\widehat{m}}$.\\

As mentioned earlier, the idea underlying the use of the matrices $T$ and $S$
lies in the regression models (\ref{modele_regression}). Indeed, these
regressions naturally appear when deriving the negative log-likelihood
(\ref{definition_TS}):
\begin{eqnarray*}
2\mathcal{L}_n(T',S') =\sum_{i=1}^p s'^{-1}_i\|{\bf X}_i+ {\bf
X}_{<i}(t_i')^*\|_n^2+\log(s'_i)\ ,
\end{eqnarray*}
where $\|.\|_n$ stands for the Euclidean norm in $\mathbb{R}^n$ divided by
$\sqrt{n}$.
By definition of  $\widehat{T}_m$ and $\widehat{S}_m$, we easily derive that the
$i$-th row vector $\widehat{t}_{i,m_i}$ of $\widehat{T}_m$ and the $i$-th
diagonal element $\widehat{s}_{i,m_i}$ of $\widehat{S}_m$ respectively equal
 \begin{eqnarray}\label{expression_t_s}
 \widehat{t}_{i,m_i}=\arg\min_{\text{supp}(t'_i)\subset m_i}\|{\bf X}_i+ {\bf
X}_{<i}(t_i')^*\|^2_n\hspace{0.5cm}\text{and}\hspace{0.5cm}\widehat{s}^2_{i,m_i}
=\|{\bf X}_i+{\bf X}_{<i}\widehat{t}_{i,m_i}^*\|_n^2\  ,
\end{eqnarray}
for any $1\leq i\leq p$. Here,  $\text{supp}(t'_i)$ stands for the support of
$t'_i$.
Hence, the row vector $\widehat{t}_{i,m_i}$ is the least-squares estimator of
$t_i$ in the regression model (\ref{modele_regression}) and 
$\widehat{s}_{i,m_i}$ is the empirical conditional variance of $X_i$ given
$X_{m_i}$. There are two main consequences: first, Expression
(\ref{expression_t_s}) emphasizes the connection between covariance estimation
and linear regression in a Gaussian design. Second, it highly simplifies the
computational cost of our procedure.  Indeed, the negative log-likelihood
$\mathcal{L}_n(\widehat{T}_m,\widehat{S}_m)$ now writes
\begin{eqnarray*}
 \mathcal{L}_n\left(\widehat{T}_m,\widehat{S}_m\right)= \frac{1}{2}\sum_{i=1}^p
\left[\log\left(\widehat{s}_{i,m_i}\right)+1\right] \ .
\end{eqnarray*}
and it follows that $\widehat{m}_i=
\arg\min_{m_i\in\mathcal{M}_i}\log\left(\widehat{s}_{i,m_i}\right)+pen_i(m_i)$.
This is why we suggest to compute $\widehat{m}$ and $\widehat{\Omega}$ as
follows. Assume we are given a collection of graphs
$\mathcal{M}=(\mathcal{M}_1,\ldots,\mathcal{M}_p)$ and penalty functions
$(pen_1(.),\ldots,pen_p(.))$. 

\bigskip

\fbox{
\begin{minipage}{0.95\textwidth}
\begin{algo}
Computation of $\widehat{m}$ and $\widetilde{\Omega}$.
\begin{enumerate}
 \item For $i$ going from $1$ to $p$,
\begin{itemize}
 \item Compute $\widehat{s}_{i,m_i}$ for each model $m_i\in\mathcal{M}_i$.
\item Take
$\widehat{m}_i=\arg\min_{m_i\in\mathcal{M}_i}\log\left(\widehat{s}_{i,m_i}
\right)+pen_i(m_i)$.
\end{itemize}
\item Set $\widehat{m}= (\widehat{m}_1,\ldots,\widehat{m}_p)$ and built
$(\widetilde{T},\widetilde{S})$ by gathering the estimators
$(\widehat{t}_{i,\widehat{m}_i},\widehat{s}_{i,\widehat{m}_i})$.
\item Take $\widetilde{\Omega}=\widetilde{T}\widetilde{S}^{-1}\widetilde{T}$.
\end{enumerate}
\end{algo}
\end{minipage}

}

\bigskip

In what follows, we refer to this method as {\bf ChoSelect}. In order to select
$\widehat{m}$, one needs to compute all $\widehat{s}_{i,m_i}$ for any
$i\in\{1,\ldots,p\}$ and any model $m_i\in\mathcal{M}_i$. Hence, the complexity
of the procedure is proportional to $\sum_{i=1}^p\left|\mathcal{M}_i\right|$. We
further discuss computational issues and we provide a faster procedure in
Section \ref{section_fast_algorith}.

\section{Risk analysis}\label{section_theorique}

In this section, we first provide a bias-variance decomposition for the Kullback
risk of the parametric estimator $\widehat{\Omega}_m$. Afterwards, we state a
general non-asymptotic risk bound for $\widetilde{\Omega}$.

\subsection{Parametric estimation}\label{section_decomposition}

Let $m$ be model in $\mathcal{M}$. Let us define the matrix $\Omega_m$ as the
best approximation of $\Omega$ that corresponds to the model $m$. The matrices
$T_m$ and $S_m$ are defined as the minimizers in $\mathcal{T}_m$ and $Diag(p)$
of the Kullback loss with $\Omega$	
\begin{eqnarray*}
 \left(T_m,S_m\right):= \arg\min_{T'\in \mathcal{T}_m,\ S'\in \text{Diag}(p)}
\mathcal{K}\left(\Omega;T'^*S'^{-1}T'\right)
\end{eqnarray*}
We note $\Omega_m=T_m^{*}S_m^{-1}T_m$.

We define the conditional Kullback-Leibler divergence of the distribution of
$X_i$ given $X_{<i}$ by
\begin{eqnarray}\label{definition_kullback_conditionnel}
\mathcal{K}\left(t_i,s_i;t'_i,s_i'\right):=\mathbb{E}\left\{\mathcal{K}\left[
\mathbb{P}_{t_i,s_i}(X_i|X_{<i});\mathbb{P}_{t'_i,s_i'}(X_i|X_{<i})\right]
\right\} \ ,
\end{eqnarray}
where $\mathbb{P}_{t_i,s_i}(X_i|X_{<i})$ stands for the conditional distribution
of $X_i$ given $X_{<i}$ with parameters $(t_i,s_i)$.
Applying the chain rule, we  obtain that
$\mathcal{K}(\Omega;\Omega')=\sum_{i=1}^p\mathcal{K}\left(t_i,s_i;t'_i,
s_i'\right)$. Consequently, we analyze the Kullback risk
$\mathbb{E}[\mathcal{K}(\Omega;\widehat{\Omega}_m)]$ by controlling each 
conditional risk
$\mathbb{E}\left[\mathcal{K}(t_i,s_i;\widehat{t}_{i,m_i},\widehat{s}_{i,m_i}
)\right]$. Let us define $t_{i,m_i}$ and $s_{i,m_i}$ as the \emph{projections}
of $(t_i,s_i)$ on the space associated to the model $m_i$ with respect to the
Kullback divergence $\mathcal{K}(t_i,s_i;.,.)$. In other words, $t_{i,m_i}$ and
$s_{i,m_i}$ satisfy
\begin{eqnarray*}
 t_{i,m_i}=\arg\min_{\text{supp}(t'_i)\subset
m_i}\mathbb{E}\left[\left(X_i+X_{<i}(t_i')^*\right)^2\right]\hspace{0.5cm}\text{
and}\hspace{0.5cm}s_{i,m_i}=\var\left(X_i|X_{<i}\right)\ .
\end{eqnarray*}
Applying the chain rule, we check that $t_{i,m_i}$ corresponds to $(i-1)$-th
first elements of the $i$-th line of $T_m$ and $s_{i,m_i}$ is the $i$-th
diagonal element of $S_m$.
Thanks to the previous  property, we derive a bias-variance decomposition for
the Kullback risk
$\mathbb{E}\left[\mathcal{K}(t_i,s_i;\widehat{s}_{i,m_i},\widehat{s}_{i,m_i}
)\right]$.

\begin{prte}\label{prte_basic_kullback_mle}
Assume that $|m_i|$ is smaller than $n-2$. The Kullback risk of
$(\widehat{t}_{i,m_i},\widehat{s}_{i,m_i})$ decomposes as follows 
\begin{eqnarray}
\mathbb{E}\left[\mathcal{K}\left(t_i,s_i;\widehat{t}_{i,m_i},\widehat{s}_{i,m_i}
\right)\right] = \mathcal{K}\left(t_i,s_i;t_{i,m_i},s_{i,m_i}\right)+R_{n,|m_i|}
\ ,\label{decomposition_biais_variance}
\end{eqnarray}
where  $R_{n,d}$ is defined by
\begin{eqnarray*}
R_{n,d}:=\frac{d+1}{n-d-2}+
\frac{d(d+1)}{2(n-d-1)(n-d-2)}+\frac{1}{2}\left[
\Psi\left(n-d\right)+\log\left(1-\frac{d}{n}\right)\right] \ ,
\end{eqnarray*}
and $\Psi(n-d):=\mathbb{E}\left[\log\left(\chi^2(n-d)/(n-d)\right)\right]$.
Besides, $R_{n,d}$ is bounded as follows 
\begin{eqnarray*}
\frac{d+1}{2(n-d-2)}\leq R_{n,d}\leq
\frac{d+1}{n-d-2}+\frac{1}{2}\left[\frac{d+1}{n-d-2}\right]^2\\ \text{and
}R_{n,d} =\frac{d+1}{2(n-d-2)}+\mathcal{O}\left(\frac{d+1}{n}\right)^2\ .
\end{eqnarray*}
\end{prte}
An explicit expression of $R_{n,d}$ is provided in the proof. Applying the chain
rule, we then  derive a bias-variance decomposition for the maximum likelihood
estimator $\widehat{\Omega}_m$.

\begin{cor}\label{corollaire_risque}
Let $m=(m_1,\ldots,m_p)$ be a model such that the size $|m_i|$ of each submodel
is smaller than $n-2$. Then, the Kullback risk of the maximum likelihood
estimator $\widehat{\Omega}_m$ decomposes into
 \begin{eqnarray*}
\mathbb{E}\left[\mathcal{K}\left(\Omega;\widehat{\Omega}_m\right)\right]
=\mathcal{K}\left(\Omega;\Omega_m\right)+\sum_{i=1}^pR_{n,|m_{i}|}\ .
\end{eqnarray*}
\end{cor}
If the size $|m_i|$ of each submodels is small with respect to $n$, the variance
term 
is of the order $\sum_{i=1}^p(|m_i|+1)/[2(n-|m_i|-2)]$. For other loss functions
such as the Frobenius norm or the $l_2$ operator norm between $\Omega$ and
$\widehat{\Omega}_m$, there is no such bias-variance decomposition with a
variance term that does not depend on the target.

\subsection{Main result}\label{section_main_result}

In this subsection, we state a general non-asymptotic oracle inequality for the
Kullback-Leibler risk of the estimator $\widetilde{\Omega}$. We shall consider
two types of penalty function $pen(.)$: the first one only takes into account
the complexity of the model collection while the second is based on a prior
probability on the model collection.

\begin{defi}\label{complexite_modele}
For any integer $i$ between $2$ and $p$, the complexity function $H_i(.)$ is
defined by
\begin{eqnarray*}
H_i(d) := \frac{1}{d}\log\left|\left\{m\in\mathcal{M}_i,\, |m_i|=d
\right\}\right|\ ,
\end{eqnarray*}
where $d$ is  any integer larger or equal to $1$. Besides, $H_i(0)$ is set to
$0$ for any $i$ between $1$ and $p$.
\end{defi}
These functions are analogous to the complexity measures introduced in
\cite{massart_pente} Sect.1.3 or in \cite{verzelen_regression} Sect.3.2. We
shall obtain an oracle inequality for complexity-based penalties under the
following assumption.~\\~\\
\textit{Assumption} $(\boldsymbol{\mathbb{H}}_{K,\eta})$: Given $K>1$ and
$\eta>0$, the collection $\mathcal{M}$ and the number $\eta$ satisfy
\begin{eqnarray}\label{Hypothese_HK}
\,\forall\ 2\leq i\leq p\ , \forall m_i\in\mathcal{M}_i\ ,\,\,\,\,\,\,
\frac{|m_i|}{n-|m_i|}\left[1+\sqrt{2H_i(|m_i|)}\right]^2\leq \eta<\eta(K)\ ,
\end{eqnarray}
where $\eta(K)$ is defined as
 $\eta(K):= [1-2(3/(K+2))^{1/6}]^2\bigvee [1-(3/K+2)^{1/6}]^2/4$. The
function $\eta(.)$ is positive and increases to one with $K$. 
This condition requires that the size of the collection is not too large.
Assumption ($\mathbb{H}_{K,\eta}$) is similar to the assumption made in
\cite{verzelen_regression} Sect 3.1 for obtaining an oracle inequality in the
linear regression with Gaussian design framework.
We further discuss $(\mathbb{H}_{K,\eta})$ in Sections \ref{section_ordered} and
\ref{section_complete} when considering the particular problems of  ordered and
complete variable selection.

\begin{thrm}\label{main_thrm}
 Let $K>1$ and let $\eta<\eta(K)$. Assume that $n$ is larger than some quantity
$n_0(K)$ only depending on $K$ and that the collection $\mathcal{M}$ satisfies
$(\boldsymbol{\mathbb{H}}_{K,\eta})$. 
If the penalty	$pen(.)$ is lower bounded as follows  
\begin{eqnarray}\label{penalite_complexe}
pen_i(m_i)\geq K\frac{|m_i|}{n-|m_i|}\left(1+\sqrt{2H_i(|m_i|)}\right)^2
\; \text{for any $1\leq i\leq p$ and  $m_i\in \mathcal{M}_i$}\ ,
\end{eqnarray}
then the risk of $\widetilde{\Omega}$ is upper bounded by
\begin{eqnarray}\label{inegalite_oracle}
 \mathbb{E}\left[\mathcal{K}\left(\Omega;\widetilde{\Omega}\right)\right]\leq
L_{K,\eta}\inf_{m\in\mathcal{M}}\left[\mathcal{K}
\left(\Omega;\Omega_m\right)+pen(m)+\frac{p}{n}\right]+ \tau_n\ ,
\end{eqnarray}
where $\tau_n$ is defined by
$$\tau_n=\tau\left(\Omega,K,\eta,n,p\right):=L_{K,\eta}n^{5/2}\left[p+\mathcal{K
}
(\Omega;I_p)\right]\exp\left[-nL_2(K,\eta)\right]\ ,$$ and $L_2(K,\eta)$ is
positive. Here, $I_p$ stands for the identity matrix of size $p$.
\end{thrm}

\begin{remark}
 This theorem tells us that $\widetilde{\Omega}$ performs almost as well as the
best trade-off between the bias term $\mathcal{K}(\Omega;\Omega_m)$ and the
penalty term $pen(m)$. The term $p/n$ is unavoidable since it is of the same
order as the variance term for the null model by Corollary
\ref{corollaire_risque}. The error term $\tau_n$ is considered as negligible
since converges exponentially fast to $0$ with $n$.
\end{remark}

\begin{remark}
 The result is non-asymptotic and holds for arbitrary large $p$ as longs $n$ is
larger than the quantity $n_0(K)$ (independent of $p$). There is no hidden
dependency on $p$ except in the complexity functions $H_i(.)$  and Assumption
$(\mathbb{H}_{K,\eta})$ that we shall discuss for particular cases in Sections
\ref{section_oracle_ordonne} and \ref{section_oracle_complet}. Besides, we are
not performing any assumption on the true precision matrix $\Omega$ except that
it is invertible. In particular, we do not assume that it is sparse and we give
a rate of convergence that only depends on a bias variance trade-off. Besides,
there is no hidden constant that depends on $\Omega$ (except for $\tau_n$). 
\end{remark}

\begin{remark}
Finally, the penalty introduced in this theorem only depends on the collection
$\mathcal{M}$ and on a number $K>1$. One chooses the parameter $K$ depending on
how conservative  one wants the procedure to be. We further discuss the
practical choice of $K$ in Sections \ref{section_ordered} and
\ref{section_complete}. In any case, the main point is that we do not need any
additional method to calibrate the penalty. 
\end{remark}

\subsection{Penalties based on a prior distribution}\label{section_prior}

The penalty defined in Theorem \ref{main_thrm} only depends on the models
through their cardinality. However, the methodology developed in the proof
easily extend to the case where the user has some \emph{prior} knowledge of the
relevant models. \\

Suppose we are give a prior probability measure
$\pi_{\mathcal{M}}=\pi_{\mathcal{M}_1}\times \ldots \times \pi_{\mathcal{M}_p}$
on the collection $\mathcal{M}$.  For any non-empty model $m_i\in\mathcal{M}_i$,
we define $l^{(i)}_{m_i}$ by 
\begin{eqnarray}\label{Hypothese_HKl}
\, \forall\ 2\leq i\leq p\ , \forall m_i\in\mathcal{M}_i\ ,\,\,\,\,\,\,
l^{(i)}_{m_i}:=-\frac{\log\left(\pi_{\mathcal{M}_i}(m_i)\right)}{|m_i|}\ .
\end{eqnarray}
By convention, we set $l^{(i)}_{\emptyset}$ to $1$. We define in the next
proposition penalty functions based on the quantity $l^{(i)}_m$ that allow to
get non-asymptotic oracle inequalities.~\\~\\
%
%
%
%
%
%
\textit{Assumption} $(\mathbb{H}^{\text{\emph{bay}}}_{K,\eta})$: Given $K>1$ and
$\eta>0$, the collection $\mathcal{M}$, the numbers $l^{(i)}_m$ and the number
$\eta$ satisfy
\begin{eqnarray}\label{Hypothese_HK2}
\,\forall\ 2\leq i\leq p\ , \forall m_i\in\mathcal{M}_i\ ,\,\,\,\,\,\,
\frac{|m_i|}{n-|m_i|}\left[1+\sqrt{2l^{(i)}_{m_i}}\right]^2\leq \eta<\eta(K)\ ,
\end{eqnarray}
where $\eta(K)$ is defined as in $(\mathbb{H}_{K,\eta})$.

\begin{prte}\label{proposition_priori}
Let $K>1$ and let $\eta<\eta(K)$. Assume that $n\geq n_0(K)$ and that Assumption
$(\mathbb{H}^\text{bay}_{K,\eta})$ is fulfilled. If the penalty	$pen(.)$ is
lower bounded as follows  
\begin{eqnarray}\label{penalite_complexe_bay}
pen_i(m_i)\geq K\frac{|m_i|}{n-|m_i|}\left(1+\sqrt{2l^{(i)}_{m_i}}\right)^2
\, \, \, \text{for any $1\leq i\leq p$ and  any $m_i\in \mathcal{M}_i$}\ ,
\end{eqnarray}
then the risk of $\widetilde{\Omega}$ is upper bounded by
\begin{eqnarray}\label{inegalite_oracle_bay}
 \mathbb{E}\left[\mathcal{K}\left(\Omega;\widetilde{\Omega}\right)\right]\leq
L_{K,\eta}\inf_{m\in\mathcal{M}}\left[\mathcal{K}
\left(\Omega;\Omega_m\right)+pen(m)+\frac{p}{n}\right]+ \tau_n\ ,
\end{eqnarray}
where $L_{K,\eta}$ and $\tau_n$ are the same as in Theorem \ref{main_thrm}.
\end{prte}

The proof is postponed to the technical Appendix~\cite{technical}.

\begin{remark}
In this proposition, the penalty (\ref{penalite_complexe_bay}) as well as the
risk bound (\ref{inegalite_oracle_bay}) depend on the prior distribution
$\pi_{\mathcal{M}}$. In fact, the bound (\ref{inegalite_oracle_bay}) means that
$\widetilde{\Omega}$ achieves the trade-off between the bias and some prior
weight, which is of the order $-\log[\pi_{\mathcal{M}}(m)]/n\ .$
This emphasizes that $\widetilde{\Omega}$ favours models with a high prior
probability. Similar risk bounds are obtained in the fixed design regression
framework in Birg\'e and Massart~\cite{birge98}. 
\end{remark}

\begin{remark}
Roughly speaking,  Assumption ($\mathbb{H}_{K,\eta}^\text{\emph{bay}}$) requires
that the prior probabilities $\pi_{\mathcal{M}_i}(m_i)$ are not exponentially
small with respect to $n$.
 \end{remark}

\section{Adaptive banding}\label{section_ordered}

In this section, we apply our method ChoSelect to the adaptive banding problem
and we investigate its theoretical properties.

\subsection{Oracle inequalities}\label{section_oracle_ordonne}

Let $d$ be some fixed positive integer which stands for the largest dimension of
the models $m_i$.
For any $2\leq i\leq p$, we consider the ordered collections
$$\mathcal{M}_{i,\text{ord}}^{d}:=\left\{\emptyset,\{1\},\{1,2\},\ldots,\{
1\wedge (i-d),\ldots, i-1\} \right\}\ ,$$
and $\mathcal{M}_{1,\text{ord}}^{d}:=\{\emptyset\}$. A model
$m=\left(\emptyset,\ldots , \{1,\ldots, k_i\}, \ldots , \{1,\ldots,
k_p\}\right)$ in the collection $\mathcal{M}_{\text{ord}}^{d}$ corresponds to
the set of matrices $T$ such that on each line $i$ of $T$, only the $k_i$
closest entries to the diagonal are possibly non-zero. This collection of models
is suitable when the matrix $T$ is approximately banded.\\

For any $1\leq i\leq p$ and any model $m_i$ in $\mathcal{M}_{i,\text{ord}}^d$ we
fix the penalty 
\begin{eqnarray}\label{penalite_ordonne}
pen_i(m_i)=K\frac{|m_i|}{n-|m_i|}\ .
\end{eqnarray}
We write $\widetilde{\Omega}^d_{\text{ord}}$ for the estimator
$\widetilde{\Omega}$ defined with the collection $\mathcal{M}^d_{\text{ord}}$
and the penalty (\ref{penalite_ordonne}).
\begin{cor}\label{cor_oracle_ordonne}
Let $K>1$, $\eta$ smaller than $\eta(K)$. Assume that $d\leq
n\frac{\eta}{1+\eta}$.
 If $n$ is larger than some quantity $n_0(K)$, then  
\begin{eqnarray}\label{oracle_ordonne}
 \mathbb{E}\left[\mathcal{K}\left(\Omega;\widetilde{\Omega}^d_{\text{\emph{ord}}
}\right)\right]\leq
L_{K,\eta}\inf_{m\in\mathcal{M}^d_{\text{\emph{ord}}}}\mathbb{E}\left[\mathcal{K
}
\left(\Omega;\widehat{\Omega}_m\right)\right]+
\tau_n\left(\Omega,K,\eta,n,p\right)\ .
\end{eqnarray}
\end{cor}
This bound is a direct application of Theorem \ref{main_thrm}.

\begin{remark}
 The term $\tau_n$ is defined in Theorem \ref{main_thrm} and is considered as
negligible since it converges to $0$ exponentially fast towards $0$. Hence, the
penalized estimator $\widetilde{\Omega}$ achieves an oracle inequality
\emph{without} any assumption on the target $\Omega$.
\end{remark}

\begin{remark}
This oracle inequality is non-asymptotic and holds for any $p$ and any $n$
larger
than $n_0(K)$. Moreover, by choosing a constant $K$ large enough, one can
consider a maximal dimension of model $d$ up to the order of $n$, because
$\eta(K)$ converges to one when $K$ increases.
\end{remark}

\noindent
 {\it Choice of the parameters $K$ and $d$.} Setting $K$ to $2$ gives a
criterion close to $AICc$ (see for instance \cite{MacQuarrie}). Besides,
Verzelen \cite{verzelen_regression} (Prop.3.2)  has justified in a close
framework this choice of $K$ is asymptotically optimal. A choice of $K=3$ is
advised if one wants a more conservative procedure. We have stated Corollary
\ref{cor_oracle_ordonne} for models $m_i$ of size  smaller than
$d=\frac{\eta}{1+\eta}n$. In practice, taking  the size  $n/2$ yields rather
good results even if it is not completely ensured by the theory. 

~\\
{\it Computational cost}. The procedure is fast in this setting. Indeed, its
complexity is the same as $p$ times the complexity of an ordered variable
selection in a classical regression framework. From numerical comparisons, it
seems to be slightly faster than the methods of Bickel and
Levina~\cite{bickel08} and Levina et al.~\cite{levina08} which require 
cross-validation type strategies.

\subsection{Adaptiveness with respect to
ellipsoids}\label{section_minimax_ordonne}

We now state that the estimator $\widetilde{\Omega}^d_{\text{ord}}$ is
simultaneously minimax over a large class of sets that we call ellipsoids.

\begin{defi}\label{definition_ellipsoide} 
Let $(a_i)_{1\leq i\leq p-1}$ be a non-increasing sequence of positive numbers
such that $a_1=1$ and let $R$ be a positive number. Then, the set
$\mathcal{E}(a,R,p)$ is made of all the non-singular matrices
$\Omega=T^*S^{-1}T$ where $S$ is in $Diag(p)$ and $T$ is a lower triangular
matrix with unit diagonal that satisfies the following property
\begin{eqnarray}\label{definition_ellipsoide_equation}
\sum_{j=1}^{i-1} \frac{T[i,i-j]^2}{a_j^2}\leq R^2\ ,\hspace{1cm}\forall\,  2\leq
i\leq p\ .
\end{eqnarray}
\end{defi}
By convention, we set $a_p=0$. The sequence $(a_i)$ measures the rate of decay
of each line of $T$ when one moves away the diagonal. Observe that in this
definition, every line of $T$ decreases the same rate. To the price of more
technicity, we can also allow different rates of decay for each line of $T$. We
shall restrict ourselves to covariance matrices with eigenvalues that lie in a
compact when considering the ellipsoid $\mathcal{E}(a,R,p)$
\begin{eqnarray}\label{definition_gamma}
 \mathcal{B}_{\text{op}}(\gamma) :=
\left\{\varphi_{\text{min}}\left(\Omega\right)\geq \frac{1}{\gamma}\text{ and
}\varphi_{\text{max}}\left(\Omega\right)\leq \gamma\right\}\ .
\end{eqnarray}

\begin{prte}\label{minoration_ellipsoide}
For any ellipsoid $\mathcal{E}(a,R,p)$, the minimax rates of estimation is lower
bounded by  
\begin{eqnarray}\label{minoration_ellipsoide_equation}
 \inf_{\widehat{\Omega}}\sup_{\Omega\in\mathcal{E}(a,R,p)}\mathbb{E}\left[
\mathcal{K}\left(\Omega;\widehat{\Omega}\right)\right]\geq
Lp\sup_{k=1,\ldots,\lfloor\sqrt{n}\rfloor}\left(R^2a^2_{k}\wedge\frac{k+1}{n}
\right)\ .
\end{eqnarray}
Let us consider the estimator $\widetilde{\Omega}^d_{\text{ord}}$ defined in
Section \ref{section_oracle_ordonne} with $d=\lfloor n\frac{\eta}{1+\eta}\rfloor
$ and the penalty (\ref{penalite_ordonne}). We also fix $\gamma>2$. If the
sequence $(a_i)_{1\leq i\leq p}$ and $R$ also satisfy $R^2\geq \frac{1}{n}$ and
$a_{\lfloor \sqrt{n}\rfloor \wedge p}^2\leq \frac{1}{R^2\sqrt{n}}$, then 
\begin{eqnarray}\label{adaptation_ellipsoide_equation}
\sup_{\Omega\in\mathcal{E}(a,R,p)\cap\mathcal{B}_{op}(\gamma)}\mathbb{E}\left[
\mathcal{K}\left(\Omega;\widetilde{\Omega}^d_{\text{\emph{Co}}}\right)\right]
\leq 
L_{K,\eta,\beta,\gamma}\inf_{\widehat{\Omega}}\sup_{\Omega\in\mathcal{E}(a,R,
p)\cap
\mathcal{B}_{op}(\gamma)}\mathbb{E}\left[\mathcal{K}\left(\Omega;\widehat{\Omega
}\right)\right] \ ,
\end{eqnarray}
if $n$ is larger than $n_0(K)$
\end{prte}

\begin{remark}
 The minimax rates of convergence over $\mathcal{E}(a,R,p)$ in the lower bound
(\ref{minoration_ellipsoide_equation}) is similar to the one obtained for
classical ellipsoids in the Gaussian fixed design regression setting (see for
instance \cite{massartflour} Th. 4.9). We conclude from the second result that
our estimator $\widetilde{\Omega}^d_{\text{ord}}$ is minimax adaptive to the
ellipsoids that are not degenerate (i.e. $R^2\geq 1/n$) and whose rates $(a_i)$
does not converge too slowly towards zero (i.e. $a_{\lfloor \sqrt{n}\rfloor
\wedge p}^2\leq (R^2\sqrt{n})^{-1}$). Note that all the sequences $(a_i)$ such
that $a_i^2\leq R^2/i$ satisfy the last assumption. 
\end{remark}

\begin{remark}
 
 However, the estimator $\widetilde{\Omega}^d_{\text{ord}}$ is not adaptive to
the parameter $\gamma$ since the constant $L$ in
(\ref{adaptation_ellipsoide_equation}) depends on $\gamma$. This is not really
surprising. Indeed,  the oracle inequality (\ref{oracle_ordonne}) is expressed
in terms of the Kullback loss while the ellipsoids are defined in terms of the
entries of $T$. If we would have considered the minimax rates of estimation over
sets analogous to $\mathcal{E}(a,R,p)$ but defined in terms of the decay of the
Kullback bias, then we would have obtained  minimax adaptiveness without any
condition on the eigenvalues. 

\end{remark}

 We are also able to prove asymptotic rates of convergence and asymptotic
minimax properties with respect to  the Frobenius loss function. For any $s>0$,
we define the ellipsoid $\mathcal{E}'(s,p,R)$ as the ellipsoid
$\mathcal{E}(a,R,p)$ with the sequence $(a_i)_{1\leq i \leq p-1}:= i^{-s}$.

\begin{cor}\label{minimax_ordonnee_Frobenius} 
If $\sum_{i=1}^{p_n}k_i+p_n=o(n)$ and $k:=1\vee \max_{1\leq i\leq p}k_i$ is
smaller than $\sqrt{n}$ then 
uniformly over the set
$\mathcal{U}_{\text{\emph{ord}}}[(k_1,\ldots,k_{p_n}),+\infty]\cap\mathcal{B}_{
\text{op}}(\gamma)$, 
\begin{eqnarray}\label{majoration_Frobenius_ordonnee}
\|\Omega-\widetilde{\Omega}^d_{\text{\emph{ord}}}\|_F^2=
\mathcal{O}_P\left(\frac{\sum_{i=1}^{p_n}k_i+p_n}{n}\right) 
\end{eqnarray}
If $s>1/2$, then uniformly over the set 
$\mathcal{E}'(s,R,p_n)\cap\mathcal{B}_{op}(\gamma)$, the estimator
$\widetilde{\Omega}^d_{\text{\emph{ord}}}$ satisfies
\begin{eqnarray}\label{majoration_Frobenius_ellipsoide}
 \|\Omega-\widetilde{\Omega}^d_{\text{\emph{ord}}}\|_F^2=\mathcal{O}_P\left[
p_n\left(\left(\frac{R}{n^s}\right)^{\frac{2}{2s+1}}\wedge
\frac{p_n}{n}\right)\right]\ .
\end{eqnarray}
Moreover, these two rates are optimal from a minimax point of view.
\end{cor}

The estimator $\widetilde{\Omega}_{\text{ord}}^d$ achieves the minimax rates of
estimation over special cases of ellipsoids. However, all these results depend
on $\gamma$ and are of \emph{asymptotic} nature.

\section{Complete graph selection}\label{section_complete}

We now turn to the complete Cholesky factor estimation problem. First, we adapt
the model selection procedure ChoSelect to this  setting. Then, we derive an
oracle inequality for the Kullback loss. Afterwards, we state that the procedure
is minimax adaptive  to the unknown sparsity both with respect to the Kullback
entropy and the Frobenius norm. Finally, we discuss the computational complexity
and we introduce a faster two-stage procedure.

\subsection{Oracle inequalities} \label{section_oracle_complet}

Again, $d$ is a positive integer that stands for the maximal size of the models
$m_i$. We consider the collections of models  $\mathcal{M}_{i,\text{co}}^d$ that
contain all the subsets of $\{1,\ldots,i-1\}$ of size smaller or equal to $d$. A
model $m\in \mathcal{M}_{\text{co}}^d$ corresponds to a pattern of zero in the
Cholesky factors $T$. As explained in Section \ref{section_preliminaire}, such a
model $m$ is also in correspondence with an ordered graph $\overrightarrow{G}$
which is compatible with the ordering. Hence, the collection
$\mathcal{M}_{\text{co}}^d$ is in correspondence with the set of ordered graphs
$\overrightarrow{G}$ of degree smaller than $d$ which are compatible with the
natural ordering of $\{1,\ldots,p\}$.

For any $2\leq i\leq p$ and any model $m_i$ in $\mathcal{M}_{i,\text{co}}^d$ we
fix the penalty  
\begin{eqnarray}\label{penalite_complete_log}
pen_i(m_i)=\log\left[1+K\frac{|m_i|}{n-|m_i|}\left\{1+\sqrt{2\left[
1+\log\left(\frac{i-1}{|m_i|}\right)\right]}\right\}^2 \right]\ ,
\end{eqnarray}
where $K>1$. In the sequel, $\widetilde{\Omega}^d_{\text{co}}$ corresponds to
the estimator ChoSelect with the collection $\mathcal{M}^d_{\text{co}}$ and the
penalty (\ref{penalite_complete_log}).

\begin{cor}\label{cor_oracle_complet}
Let $K>1$ and $\eta<\eta'(K)$ (defined in the proof). Assume that 
\begin{eqnarray}\label{dimension_complete}
 d\leq \eta \frac{n}{1+[\log(p/d)\vee 0]} \ .
\end{eqnarray}
If $n$ is larger than some quantity $n_0(K)$, then
$\widetilde{\Omega}^d_{\text{co}}$ satisfies 
\begin{eqnarray}
 \mathbb{E}\left[\mathcal{K}\left(\Omega;\widetilde{\Omega}^d_{\text{\emph{co}}}
\right)\right]& \leq & 
L_{K,\eta}\inf_{m\in\mathcal{M}^d_{\text{\emph{co}}}}\left\{\mathcal{K}
\left(\Omega;\Omega_m\right)+\sum_{i=2}^p
\frac{|m_i|}{n-|m_i|}\left[1+\log\left(\frac{i-1}{|m_i|}\right)\right]+\frac{p}{
n}\right\} \nonumber\\ 
& + & \tau'_n\ ,\label{oracle_complet}
\end{eqnarray}
where the remaining term $\tau'_n$ is of the same order as $\tau_n$ in Theorem
\ref{main_thrm}.
\end{cor}
A proof is provided in Section \ref{section_preuve_oracle}.
We get an oracle inequality up to logarithms factors,  but we prove in Section
\ref{section_minimax_complete} that these terms $\log[(i-1)/|m_i|]$ are in fact
unavoidable.
For the sake of clarity, we straightforwardly derive from (\ref{oracle_complet})
the less sharp but more readable upper bound 
$$\mathbb{E}\left[\mathcal{K}\left(\Omega;\widetilde{\Omega}^d_{\text{co}}
\right)\right] \leq  
L_{K,\eta}\inf_{m\in\mathcal{M}_{\text{co}}^d}\left\{\mathcal{K}
\left(\Omega;\Omega_m\right)+\frac{p+|m|\log
p}{n}\right\}+\tau_n\left(\Omega,K,\eta,n,p\right)\ , $$
where $|m|:= \sum_{i=1}^p|m_i|$. 

\begin{remark}
 As for the previous results, we do not perform any assumption on the target
$\Omega$ and the obtained upper bound is non-asymptotic. By Condition
(\ref{dimension_complete}), we can consider dimension $d$ up to the order
$n/[\log(p/n)\vee 1]$. If $p$ is much larger than $n$, the maximal dimension has
to be smaller than the order $n/\log(p)$. This is not really surprising since it
is also the case for linear regression with Gaussian design as stated in
\cite{verzelen_regression} Sect. 3.2. There is no precise results that proves
that this $n/\log(p)$ bound is optimal but we believe that it is unimprovable.
If $p$ is of the same order as $n$, it is possible to consider dimensions up to
the same order as $p$.
\end{remark}

\begin{remark}
The same bound (\ref{oracle_complet}) holds if we use the penalty
$$pen'_i(m_i)=K\frac{|m_i|}{n-|m_i|}\left\{1+\sqrt{2\left[1+\log\left(\frac{i-1}
{|m_i|}\right)\right]}\right\}^2\ .$$
For a given $K$, observe that $pen_i(m_i)=\log(1+pen'_i(m_i))$. Hence, these two
penalties are equivalent when $n$ is large. In Corollary
\ref{cor_oracle_complet}, we have privileged a logarithmic penalty, because this
penalty gives slightly better results in practice. 
\end{remark}

{\it Choice of $K$ and $d$}.
 In practice, we set the maximal dimension to $n/\{2.5[2+(\log(p/n)\vee 0)]\}$.
Concerning the choice of $K$, we advise to use the value $1.1$, if the goal is 
to minimize risk. When the goal is to estimate the underlying graph, one should
use a larger value of $K$ like $2.5$ in order to decrease the proportion of
falsely discovered vertices.

\subsection{Adaptiveness to unknown sparsity}\label{section_minimax_complete}

In this section, we state that the estimator $\widetilde{\Omega}_{\text{co}}^d$
achieves simultaneously the minimax rates of estimation for sparsity of the
matrix $T$. In the sequel, $\mathcal{U}_1[k,p]$ stands for the set of positive
square matrices $\Omega=T^*S^{-1}T$ of size $p$ such that its Cholesky factor
$T$ contains at most $k$ non-zero off-diagonal coefficients on each line. The
set $\mathcal{U}_1[k,p]$ contains the precision matrices of  the directed
Gaussian graphical models whose underlying directed acyclic graph
$\overrightarrow{\mathcal{G}}$ satisfies the two following properties:
\begin{itemize}
 \item It is compatible with the ordering on the variables. 
\item Each node of $\overrightarrow{\mathcal{G}}$ has at most $k$ parents.
\end{itemize}
We shall also consider the set $\mathcal{U}_2[k,p]$ that contains positive
square matrices whose whose Cholesky factor is $k$-sparse (i.e. contains at most
$k$ non-zero elements). Hence, the set  $\mathcal{U}_2[k,p]$ corresponds to the
precision matrices of the directed Gaussian graphical models whose underlying
directed acyclic graph $\overrightarrow{\mathcal{G}}$ is compatible with the
ordering on the variables and has at most $k$ edges. When $\Omega$ belongs to
$\mathcal{U}_2[k,p]$ with $k$ ``small'', we say that the underlying Cholesky
factors $T$ are ultra-sparse.

For deriving the minimax rates of estimation, we shall restrict ourselves to
precision matrices whose Kullback divergence with the identity is not too large.
This is why we define
\begin{eqnarray*}
 \mathcal{B}_{\mathcal{K}}(r):=\left\{\Omega\,\text{ s.t.
}\mathcal{K}(\Omega;I_p)\leq pr \right\}\ ,
\end{eqnarray*}
 for any positive number $r>0$.

\begin{prte}\label{corolaire_adaptation}
Let $k$ and $p$ be two positive integers such that $k\leq p$.  The minimax rates
of estimation over the sets $\mathcal{U}_1[k,p]$ and $\mathcal{U}_2[k,p]$ are
lower bounded as follows
\begin{eqnarray}\label{minoration_kullback_selection_complete}
 \inf_{\widehat{\Omega}}\sup_{\Omega\in
\mathcal{U}_1[k,p]}\mathbb{E}_{\Omega}\left[\mathcal{K}\left(\Omega;\widehat{
\Omega}\right)\right]&\geq& Lkp\frac{1+\log\left(p/k\right)}{n} \ , \quad \text{
if } n\geq Lk^2[1+\log(p/k)]\ , \\
\label{minoration_selection_complete_ultra_sparse}
 \inf_{\widehat{\Omega}}\sup_{\Omega\in
\mathcal{U}_2[k,p]}\mathbb{E}_{\Omega}\left[\mathcal{K}\left(\Omega;\widehat{
\Omega}\right)\right]&\geq& L\frac{p+k\log(p)}{n} \ ,\quad \text{ if } k\leq p.
\end{eqnarray}

 Consider $K>1$, $\beta>1$, and $\eta<\eta(K)$. Assume that $n\geq n_0(K)$ and
choose  a positive integer $d$ that satisfies Condition
(\ref{dimension_complete}). 
The penalized estimator $\widetilde{\Omega}^d_{\text{co}}$ defined in Corollary
\ref{cor_oracle_complet} is minimax adaptive over the sets
$\mathcal{U}_1[k,p]\cap\mathcal{B}_{\mathcal{K}}(n^\beta)$ for all $k$ smaller
than $d$  that also satisfy $n\geq Lk^2(1+\log(p/k))$. It is also minimax
adaptive over $\mathcal{U}_2[k,p]\cap\mathcal{B}_{\mathcal{K}}(n^\beta)$ for all
$k$ less than $d$:
\begin{eqnarray*}
\sup_{\Omega\in
\mathcal{U}_1[k,p]\cap\mathcal{B}_{\mathcal{K}}(n^\beta)}\mathbb{E}_{\Omega}
\left[\mathcal{K}\left(\Omega;\widetilde{\Omega}_{\text{co}}^d\right)\right]\leq
L_{K,\beta,\eta} \inf_{\widehat{\Omega}}\sup_{\Omega\in
\mathcal{U}_1[k,p]\cap\mathcal{B}_{\mathcal{K}}(n^\beta)}\mathbb{E}_{\Omega}
\left[\mathcal{K}\left(\Omega;\widehat{\Omega}\right)\right]\ ,\\
\sup_{\Omega\in
\mathcal{U}_2[k,p]\cap\mathcal{B}_{\mathcal{K}}(n^\beta)}\mathbb{E}_{\Omega}
\left[\mathcal{K}\left(\Omega;\widetilde{\Omega}_{\text{co}}^d\right)\right]\leq
L_{K,\beta,\eta} \inf_{\widehat{\Omega}}\sup_{\Omega\in
\mathcal{U}_2[k,p]\cap\mathcal{B}_{\mathcal{K}}(n^\beta)}\mathbb{E}_{\Omega}
\left[\mathcal{K}\left(\Omega;\widehat{\Omega}\right)\right]\ .
\end{eqnarray*}

\end{prte}

\begin{remark}
The minimax rates of estimation over $\mathcal{U}_1[k,p]$ is of order $kp[1+\log
\left(p/k\right)]/n$. We do not think that the condition $n\geq
Lk^2[1+\log(p/k)]$ is necessary but we do not know how to remove it.
The technical condition $\mathcal{K}\left(\Omega;I_p\right)\leq pn^\beta$ is not
really restrictive. It comes from the term
$n^{5/2}\mathcal{K}(\Omega;I_p)\exp\left[-nL_{K,\eta}\right]$ in Theorem
\ref{main_thrm} which goes exponentially fast to $0$ with $n$ as long as
$\mathcal{K}(\Omega,I_p)/p$ is grows polynomially with respect to $n$. 
In conclusion, our estimator $\widetilde{\Omega}^d_{\text{co}}$ is adaptive to
the sparsity of its Cholesky factor $T$. 
\end{remark}

\begin{remark}
Let us translate the proposition in terms of directed graphical models.
The Kullback minimax rate of covariance estimation over graphical models with at
most $k$ parents by node is of the order $pk(1+\log (p/k))/n$. Moreover, 
the Kullback minimax rate of covariance estimation over graphical models with at
most $k$ vertices  is of the order $(p+k\log p)/n$. Finally,
$\widetilde{\Omega}^d_{\text{co}}$ is minimax adaptive for estimating the
distribution of a sparse directed Gaussian graphical model whose underlying
graph is unknown.  
 \end{remark}

We can also consider the rates of convergence with respect to the Frobenius norm
or the operator norm in the spirit of the results of Lam and
Fan~\cite{fan_covariance}. We recall that $\|.\|_F$ and $\|.\|$ respectively
refer to the Frobenius norm and the operator norm in the space of matrices. We
also recall that the set $\mathcal{B}_{\text{op}}(\gamma)$ is defined in
(\ref{definition_gamma}).
\begin{cor}\label{risque_Frobenius_complete}
Let $K>1$, $\eta<\eta(K)$, $\gamma>2$, and let $d$ be the largest integer that
satisfies (\ref{dimension_complete}).  If $p_nk_n[1+\log(p_n/k_n)]=o(n)$, then 
\begin{eqnarray}\label{vitesse_frobenius_complete}
\|\Omega-\widetilde{\Omega}^d_{\text{\emph{co}}}\|_F^2
&=&\mathcal{O}_P\left(k_n\left[1+\log\left(\frac{p_n}{k_n}\right)\right]\frac{
p_n } {n}\right) \ , \\ \nonumber
\|\Omega-\widetilde{\Omega}^d_{\text{\emph{co}}}\|
& =&\mathcal{O}_P\left(\sqrt{k_n\left[1+\log\left(\frac{p_n}{k_n}\right)\right]
\frac{p_n}{n}}\right) \ ,
\end{eqnarray}
 uniformly on $\mathcal{U}_1[k_n,p_n]\cap\mathcal{B}_{\text{op}}[\gamma]$.  If
$p_n+k_n\log(p_n)=o(n)$, then 
\begin{eqnarray}\label{vitesse_frobenius}
\|\Omega-\widetilde{\Omega}^d_{\text{\emph{co}}}\|_F^2
& =& \mathcal{O}_P\left(\frac{p_n+k_n\log(p_n)}{n}\right)\ ,\\ \nonumber
 \|\Omega-\widetilde{\Omega}^d_{\text{\emph{co}}}\|&
=& \mathcal{O}_P\left(\sqrt{\frac{p_n+k_n\log(p_n)}{n}}\right) \ ,
\end{eqnarray}
uniformly on $\mathcal{U}_2[k_n,p_n]\cap\mathcal{B}_{\text{op}}[\gamma]$.
Moreover, all these Frobenius rates of convergence are optimal from a minimax
point of view.
\end{cor}

\begin{remark}
The estimator  $\widetilde{\Omega}_{\text{co}}^d$  is  asymptotically minimax
adaptive to the sets $\mathcal{U}_1[k,p]\cap
\mathcal{B}_{\text{\emph{op}}}(\gamma)$ and $\mathcal{U}_2[k,p]\cap
\mathcal{B}_{\text{\emph{op}}}(\gamma)$ with respect to the Frobenius norm.
Moreover, these rates are coherent with the ones obtained by Lam and Fan in
Sect.4 of \cite{fan_covariance}. We do not think that the rates of convergence
with respect to the operator norm are sharp.
\end{remark}

\begin{remark}
 These results are  of asymptotic nature and require that $p_n$ has to be much
smaller than $n$. Besides, the upper bounds on the rates highly depend on the
largest eigenvalue $\varphi_{\text{\text{max}}}(\Omega)$.
This is why we have restricted ourselves to precision matrices whose eigenvalues
lie in the compact $[1/\gamma;\gamma]$. Nevertheless, to our knowledge all 
results in this setting suffer from the same  drawbacks. See for instance Th.11
of  Lam and Fan~\cite{fan_covariance}.
\end{remark}

\section{A two-step procedure}\label{section_fast_algorith}

The computational cost of $\widetilde{\Omega}_{\text{co}}^d$ is proportional to
the size of $\mathcal{M}^d_{i,\text{co}}$, which is of the order of $p^d$.
Hence, it becomes prohibitive when $p$ is larger than $50$. In fact,
$\widetilde{\Omega}_{\text{co}}^d$ minimizes a penalized criterion over the
collection $\mathcal{M}^d_{\text{co}}$. Nevertheless,  the collections
$\mathcal{M}^d_{i,\text{co}}$ contain an overwhelming number of models that are
clearly irrelevant. This is why we shall use a two-stage procedure. First, we
compute a subcollection of  $\mathcal{M}^d_{\text{co}}$. Then, we minimize the
penalized criterion over this subcollection.

 Suppose we are given a fast data-driven method that computes a subset
$\widehat{\mathcal{M}}_i$ of $\mathcal{M}^d_{i,\text{co}}$ for any $i$ in
$1,\ldots p$.

\bigskip

\fbox{
\begin{minipage}{0.95\textwidth}
\begin{algo}\label{algorithm_faste}
Computation of $\widehat{m}^f$ and $\widetilde{\Omega}^f$
\begin{enumerate}
 \item For $i$ going from $1$ to $p$,
\begin{itemize}
\item Compute the subcollection $\widehat{\mathcal{M}}_i$ of
$\mathcal{M}^d_{i,\text{co}}$.
\item Compute $\widehat{s}_{i,m_i}$ for each model
$m_i\in\widehat{\mathcal{M}}_i$.
\item Take $\widehat{m}^f_i:=
\arg\min_{m_i\in\widehat{\mathcal{M}}_i}\log\left(\widehat{s}_{i,m_i}\right)+
pen_i(m_i)\ .$
\end{itemize}
\item Set $\widehat{m}^f= (\widehat{m}^f_1,\ldots,\widehat{m}^f_p)$ and build
$(\widetilde{T}^f,\widetilde{S}^f)$ by gathering the estimators
$(\widehat{t}_{i,\widehat{m}^f_i},\widehat{s}_{i,\widehat{m}^f_i})$.
\item Take
$\widetilde{\Omega}^f=\widetilde{T}^f(\widetilde{S}^f)^{-1}\widetilde{T}^f$.
\end{enumerate}
\end{algo}
\end{minipage}}

\bigskip

In what follows, we refer to this method as {\bf ChoSelect}$^{\text{f}}$.
 For any $2\leq i\leq p$ and any model $m_i$ in $\mathcal{M}_{i,\text{co}}^d$,
we advise to fix the penalty as in Section \ref{section_oracle_complet}:
\begin{eqnarray*}
pen_i(m_i)=\log\left[1+K\frac{|m_i|}{n-|m_i|}\left\{1+\sqrt{2\left[
1+\log\left(\frac{i-1}{|m_i|}\right)\right]} \right\}^2\ \right] ,
\end{eqnarray*}
with $K>1$. $K=1.1$ gives good results in practice.~\\~\\

\begin{remark}
Observe that we use the same data for computing the collections 
$\widehat{\mathcal{M}}_i$ and the estimator $\widetilde{\Omega}^f$. The
estimator $\widetilde{\Omega}^f$ exhibits a small risk as long as the
collections $\widehat{\mathcal{M}}_i$ contain good models as shown by the
following proposition:
\end{remark}

\begin{prte}\label{proposition_risque_fast}
Let $m$ be a model in $\mathcal{M}^d_{\text{co}}$ and $\mathbb{A}_m$ be the
event such that $m\in\widehat{\mathcal{M}}_1\times \ldots\times
\widehat{\mathcal{M}}_p$. Under the same assumptions as Corollary
\ref{cor_oracle_complet}, it holds that
\begin{eqnarray}
 \mathbb{E}\left[\mathcal{K}\left(\Omega;\widetilde{\Omega}^f\right)\mathbf{1}_{
\mathbb{A}_m}\right]& \leq & 
L_{K,\eta}\left\{\mathcal{K}\left(\Omega;\Omega_m\right)+\sum_{i=2}^p
\frac{|m_i|}{n-|m_i|}\left[1+\log\left(\frac{i-1}{|m_i|}\right)\right]+\frac{p}{
n}\right\}\nonumber\\
& + & \tau_n\ ,\label{oracle_complet_fast}
\end{eqnarray}
where $\tau_n$ is defined in Theorem \ref{main_thrm}.

\end{prte}

\begin{remark}
Hence, under the event $\mathbb{A}_{m^*}$ where $m^*$ is the oracle model,
$\widetilde{\Omega}^f$ achieves the optimal of convergence. The estimator
achieves also a small risk as soon as any "good" model belongs to the estimated
collection. Here, "good" refers to a small Kullback risk.
Observe that it is much easier to estimate a collection
$\widehat{\mathcal{M}}_i$ that contains a "good" model than directly 
estimating a "good" model.
\end{remark}

In fact, Algorithm \ref{algorithm_faste} and Proposition
\ref{proposition_risque_fast} are generally applicable to any collection
$\mathcal{M}$ and penalties defined by (\ref{penalite_complexe}) or
(\ref{penalite_complexe_bay}).\\

The computational cost of Algorithm \ref{algorithm_faste} is directly related to
the cost of the computation of $\mathcal{M}_i$ and to the size of the
collections $\widehat{\mathcal{M}}_i$. The challenge is to design a fast
procedure providing a fairly small collection $\widehat{\mathcal{M}}_i$, which
contains relevant models with large probability. Let us describe two examples of
such a procedure.


\smallskip
\fbox{
\begin{minipage}{0.95\textwidth}
\begin{algo}\label{Algorithme_lasso} Computation of the collection
$\widehat{\mathcal{M}}_i$ by the Lasso.~\\
Let $D$ be an integer smaller than $\frac{n}{2.5[2+(\log(p/n)\vee 0)]}$ and let
$k$ be any positive integer.
\begin{enumerate}
\item Using the LARS \cite{LARS} algorithm, compute the regularization path of
the Lasso for the regression of ${\bf X}_i$ with respect to the covariates ${\bf
X}_{<i}$. 
\item Order the variables $X_{(1)},\ldots,X_{((i-1)\wedge D)}$ with respect to
their appearance in the regularization path.
\item Take $\widehat{\mathcal{M}}_i:= \mathcal{P}(X_{(1)},\ldots,X_{(k\wedge
(i-1)\wedge D)})\bigcup \text{RP}(i,D)$,\\
where $\mathcal{P}(A)$  contains all the subsets of $A$ and where
$\text{RP}(i,D)$ is the regularization path stopped at $D$ variables.
\end{enumerate}
\end{algo}
\end{minipage}}

\begin{remark}
The size of the random collection $\widehat{\mathcal{M}}_i$ increases with the
parameter $k$. Suppose that $i$ is larger than $D$. The size of
$\widehat{\mathcal{M}}_i$  is generally of the order $2^k\vee D$. The case $k=0$
corresponds to choosing the regularization path of the Lasso for
$\widehat{\mathcal{M}}_i$. The estimator $\widetilde{\Omega}^f$ then performs as
well (up to a $\log p$ factor) as the best parametric estimator with a model in
the regularization path. The collection size is fairly small, but the oracle
model may not belong to $\widehat{\mathcal{M}}_i$ with large probability. This
is especially the case is the true covariance $\Sigma$ is far from the identity
since the Lasso estimator is possibly inconsistent. In many cases, the true (or
the oracle) model is a submodel of the model selected by the Lasso with a
suitable parameter \cite{bolasso}. 
When choosing $k=D$, it is therefore  likely that the true model or a "good"
model belongs to $\widehat{\mathcal{M}}_i$.  
 \end{remark}

The regularization path of the Lasso is not necessarily increasing
\cite{LARS}. If we want that $\widehat{\mathcal{M}}$ contains all subsets of
sparse solutions of the Lasso we need to use a variant of the previous
algorithm:
\smallskip

\fbox{
\begin{minipage}{0.95\textwidth}
\begin{algo}\label{Algorithme_lasso2}
Let $D$ be an integer smaller than $\frac{n}{2.5[2+(\log(p/n)\vee 0)]}$ and let
$k$ be any positive integer.
\begin{enumerate}
\item Using the LARS \cite{LARS} algorithm, compute all the Lasso solutions for
the regression of ${\bf X}_i$ with respect to the covariates ${\bf
X}_{<i}$. 
\item For any $\lambda>0$, consider the set of $\{X_{j_1},X_{j_2}\ldots
X_{j_{s_\lambda}}\}$ of variables selected by the Lasso. If $s_\lambda>k$ we
define $A_i^\lambda = \emptyset$ while we take $A_i^\lambda =
\mathcal{P}(X_{j_1},\ldots,X_{j_{s_\lambda}})$ is $s_\lambda \leq k$.
Here, $\mathcal{P}(A)$  contains all the subsets of $A$.
\item Take $\widehat{\mathcal{M}}_i:= \cup_{\lambda>0}
A_i^\lambda\bigcup \text{RP}(i,D)$,\\
where $\mathcal{P}(A)$  contains all the subsets of $A$ and where
$\text{RP}(i,D)$
is the regularization path stopped at $D$
variables.
\end{enumerate}
\end{algo}
\end{minipage}}

\smallskip

In the following proposition, we show the ChoSelect$^f$ outperforms the Lasso
under restricted eigenvalue conditions. We consider an asymptotic setup where
$p$ and $n$ go to infinity with $p$ larger $n$.\\

\noindent {\bf ASSUMPTIONS:}
\begin{itemize}

\item {\bf (H.1)} The covariance matrix $\Sigma$ satisfies restricted eigenvalue
conditions of order $q^*>0$. 
\begin{eqnarray*}
 c_*\leq \frac{u^*\Sigma_Au}{u^*u}\leq c^*,\hspace{1cm }\forall A \text{ with
$|A|=q^*$ and $u\in\mathbb{R}^{q^*}$}\ .
\end{eqnarray*}
 Moreover, we assume that 
and $q^*\log(p)/n$ goes to $0$ when $p$ and $n$ go to infinity.

\item {\bf (H.2)} Fix some $v<1$. The vector $t_{p}$ (which corresponds to the
$p$-th line of
$T$) is $q$-sparse with some $q<\frac{n^v}{\log p}\vee \frac{n}{\log p}$. The
set of non-zero component is denoted
$m_*$. Let us set some $K>24\vee (2/(1-v))$ and define $$M_2(K,c_*)=
\frac{32}{c_*}\left[\frac{2}{3}+
\frac{112c^*}{9c_*}+\left(\frac{16c^*}{3c_*}\right)^2\right]\bigvee
\left[4(K+12)/c_*\right] .$$
For any zero-component $t_p[j]$, we have
\begin{eqnarray*}
 t_p[j]^2\geq M_2(K,c_*)\frac{q\log (p)}{n}\sigma^2\ . 
\end{eqnarray*}

\item {\bf (H.3)} Define $M_1(c_*,c^*)=2+16\frac{c^*}{c_*}$.
The quantities $q$ and $q^*$ are such that
$$M_1(c_*,c^*)q+1\leq q^*\ .$$ 
\end{itemize}

\begin{prte}\label{prop_consistance}
Consider the procedure ChoSelect$^f$ with $K$ as in {\bf (H.2)} and the penalty
(\ref{penalite_complete_log}) and the algorithm \ref{Algorithme_lasso2}. Take
$k\geq  M_1^*q$ and $D=n/\log(p)^2$. Under Assumptions
${\bf
(H.1)}$, ${\bf (H.2)}$, and ${\bf (H.3)}$
\begin{eqnarray*}
 \mathbb{P}\left[\widehat{m}^f_p= m_{*,p}\right]\rightarrow 1\ .
\end{eqnarray*}
\end{prte}

The proof of the proposition is postponed to the appendix \cite{technical}

\begin{remark}
In contrast to ChoSelect$^f$, the Lasso procedure does not consistently select
the support of $t_p$ under restricted eigenvalue
conditions \cite{zhao06,zhang08}. Observe that our assumptions ({\bf H.1}),
({\bf H.2}), ({\bf H.3}) and our result are quite similar to the ones obtained
by the stability selection method of Meinshausen and B\"uhlmann
\cite{meinhausen_stability}.
\end{remark}

\begin{remark}
Under similar conditions, one can prover that ChoSelect$^f$ selects
consistently the support of any vector $t_i$ for $n\leq i\leq p$. In order to
consistently estimate the whole pattern of zero of $T$, one needs to slightly
change the penalty (\ref{penalite_complete_log}) by replacing $(i-1)$ by
$(i-1)\vee n$. 
\end{remark}

\begin{remark}
For the sake of simplicity, we have only described two methods for building the
collection $\widehat{\mathcal{M}}$. One may also use a collection based on the
adaptive Lasso or more generally any (data-driven) collection
$\widehat{\mathcal{M}}$. 
Moreover, ChoSelect$^f$ can be interpreted as a way to tune an estimation
procedure and to merge different procedures. Suppose we are given a collection 
$\mathcal{A}$ of estimation
procedure. For any procedure $a\in\mathcal{A}$, we build a collection
$\widehat{\mathcal{M}}^a$ using the model corresponding to the estimator
$\widehat{\Omega}_a$ or using a regularization path associated to $a$ (if
possible).
If we take the collection $\widehat{\mathcal{M}}$ as the reunion of all
$\widehat{\mathcal{M}}^a$ for $a\in\mathcal{A}$, then by Proposition
\ref{proposition_risque_fast} the estimator
$\widetilde{\Omega}^f$  nearly selects the best model (from the risk
point of view) among the ones  previously selected by the procedures
$a\in\mathcal{A}$.
\end{remark}

\section{Simulation Study}\label{section_simulation}

In this section, we investigate the practical performances of the proposed
estimators. We concentrate on two applications: adaptive banding and complete
graph selection.

\subsection{Adaptive banding}\label{section_simulation_banding}

 
\subsubsection{Simulation scheme}

{\it Simulating the data}. We have used a similar scheme to Levina et
al.~\cite{levina08}. Simulations were carried out for centered Gaussian
vectors with two different precision
models. The first one has entries of the Cholesky factor exponentially decaying
as one moves away from the diagonal.
\begin{eqnarray*}
{\bf \Omega_1}: \hspace{0.3cm}T[i,j]=0.5^{|i-j]},\hspace{1cm} j<i;
\hspace{1.5cm} s_i=0.01    
\end{eqnarray*}

The second model allows different sparse structures for the Cholesky factors.
\begin{eqnarray*}
{\bf \Omega_2}:\hspace{0.3cm} &k_i \sim U(1,\lceil j/2\rceil);\hspace{1.5cm}
T[i,j]=0.5,\hspace{0.5cm} &i-k_i\leq j\leq i-1\\
&T[i,j]=0, \hspace{0.5cm} j<i-k_i;\hspace{0.5cm} s_i=0.01 &
\end{eqnarray*}
Here $U(k_1,k_2)$ denotes an integer selected uniformly at random from all
integers from $k_1$ to $k_2$. We generate from this structure for $p=30$. Levina
et al. pointed out that this structure can generate poorly conditioned
covariance matrix for larger $p$. To avoid this problem, we  divide the
variables for $p=100$ and $p=200$ into respectively $4$ and $8$ different blocks
and we generate a random structure from the random structure from the model
described above for each of the blocks.

For each of the covariance models, we generate a sample of $n=100$. We consider
three different values of $p$: 30, 100, and 200. ~\\~\\
We apply the following procedures:
\begin{itemize}
\item our procedure {\bf ChoSelect} as described in Section
\ref{section_ordered}. More precisely, we take the collection
$\mathcal{M}_{\text{ord}}^{\lfloor n/2\rfloor}$, the penalty
(\ref{penalite_ordonne}), and $K=3$.
 \item the {\bf nested Lasso} method of Levina et al.~\cite{levina08}. It
is computed with the $J_1$ penalty, while its tuning parameter is selected via
$5$-fold cross-validation based on the likelihood.  We have used the penalty
$J_1$
instead of $J_2$ for
computational reasons.
\item the {\bf banding} procedure of Bickel and Levina~\cite{bickel08}. The
tuning parameter is chosen according to Sect.5 in \cite{bickel08} with $50$
random splits.
\item the regularization method of {\bf Ledoit and Wolf}~\cite{ledoit}.
\end{itemize}

For the first covariance model  $\Omega_1$, we also compute the oracle
estimator, i.e. the parametric estimator which minimizes the Kullback risk among
all the estimators $\widehat{\Omega}_m$ with $m\in\mathcal{M}^{\lfloor
n/2\rfloor}_{\text{ord}}$. We recall that the computation of the oracle
estimator require the knowledge of the target $\Omega_1$. The performances of
this estimator are presented here as a benchmark. The experiments are repeated
$N=100$ times. In the second scheme, $N_1=10$ precision matrices are sampled and
$N_2=10$ experiments are made for each sample.

\subsubsection{Results}

In Tables \ref{tab1} and \ref{tab2}, we provide evaluations of  the Kullback
loss 
$$\mathcal{K}(\Omega;\widehat{\Omega}):=
\frac{1}{2}\left[tr(\widehat{\Omega}\Omega^{-1}
)-\log(|\widehat{\Omega}||\Omega^{-1}|)-p\right] \ ,$$ the
operator distance
$\|\widehat{\Omega}-\Omega\|$, and the operator distance between the inverses
$\|\widehat{\Omega}^{-1}-\Sigma\|$ for any of the fore-mentioned estimators. We
have chosen the Kullback loss because of its connection with discriminant
analysis. The two other loss functions are interestingly connected to the
estimation of the eigenvalues and the eigenspaces.

For the second structure, we also consider the pattern of zero estimated by our
procedure, the nested Lasso and the banding method of Bickel and Levina. More
precisely, we estimate the power (i.e. the fraction of non-zero terms in $T$
estimated as non-zero) and the FDR (i.e. the ratio of the false discoveries over
the true discoveries) in Table \ref{tab3}.\\

\begin{Table}[hptb] 
\caption{Estimation and $95\%$ confidence interval of the Kullback risk, the
operator distance risk, and the operator distance between inverses risk  for 
the first covariance model $\Omega_1$.\label{tab1}} 
\begin{center}
\begin{tabular}{|l|ccccc|}
 \hline Method & Ledoit & Banding &  Nested Lasso  & ChoSelect  & Oracle\\ 
\hline
\hline & \multicolumn{5}{c|}{ Kullback discrepancy
$\mathcal{K}(\Omega;\widehat{\Omega})$ }\\ 
\hline $p=30$ & $2.00\pm 0.05$ &$0.90\pm 0.05$ & $0.87\pm 0.02$  & $1.00\pm
0.03$ & $0.79\pm
0.02$\\
\hline $p=100$& $14.4\pm 0.5$ &$3.6\pm 0.4$  &$3.2\pm 0.1$  &$3.7\pm 0.1$
& $2.9\pm
0.1$  \\
\hline  $p=200$ & $33.4\pm 0.6$ &$9.8\pm 1.5$ & $6.4\pm 0.1$  & $7.5\pm
0.1$& $5.9\pm
0.1$ \\
\hline   & \multicolumn{5}{c|}{ Operator distance $\|\widehat{\Omega}-
\Omega\|\times 10^{-2}$ }\\ 
\hline$p=30$ &$1.86\pm 0.07$ &$1.28\pm 0.06$ & $1.18\pm 0.04$ & $1.36\pm 0.06$ &
$1.19\pm
0.04$\\
\hline $p=100$&$1.76\pm 0.09$ &$1.68\pm 0.14$ & $1.52\pm 0.06$ & $1.75\pm 0.06$
& $1.49\pm
0.05$\\
\hline  $p=200$ & $1.33\pm 0.01$ & $2.19\pm 0.22$ & $1.61\pm 0.04$ & $1.92\pm
0.06$ & $1.61\pm
0.05$  \\ 
\hline  & \multicolumn{5}{c|}{ Operator distance $\|\widehat{\Omega}^{-1}-
\Sigma\|$ }\\ 
\hline $p=30$ & $0.14\pm 0.02$ &$0.15\pm 0.02$ & $0.17\pm 0.02$  & $0.15\pm
0.02$ & $0.14\pm
0.02$\\
\hline $p=100$& $1.4\pm 0.2$ &$1.4\pm 0.2$  &$1.7\pm 0.2$  &$1.5\pm 0.2$ &
$1.4\pm 0.2$ \\
\hline  $p=200$& $5.9\pm 0.6$ & $5.6\pm 0.7$ & $6.8\pm 0.7$  & $6.5\pm 0.6$ &
$5.9\pm 0.6$ \\
\hline 
\end{tabular}
\end{center}

\end{Table}

\begin{Table}[hptb] 
\caption {Estimation and $95\%$ confidence interval of the Kullback risk, the
operator distance risk, and the operator distance between inverses risk  for 
the second covariance model $\Omega_2$.\label{tab2}}  
\begin{center}
\begin{tabular}{|l|cccc|}
\hline
Method & Ledoit & Banding &  Nested Lasso  & ChoSelect   \\ \hline
\hline& \multicolumn{4}{c|}{ Kullback discrepancy
$\mathcal{K}(\Omega;\widehat{\Omega})$ }\\ 
\hline $p=30$ & $112\pm 4$ &$3.2\pm 0.2$ & $3.2\pm 0.2$  & $1.2\pm 0.1$ \\
\hline $p=100$& $253\pm 7$ &$27.4\pm 1.6$  &$7.6\pm 0.2$  &$3.5\pm 0.1$  \\
\hline  $p=200$ & $565\pm 5$ &$58\pm 2$ & $14.6\pm 0.2$  & $7.2\pm 0.1$ \\
\hline & \multicolumn{4}{c|}{ Operator distance $\|\widehat{\Omega}-
\Omega\|\times 10^{-2}$ } \\
\hline $p=30$& $9.6\pm 0.4$ &$8.2\pm 0.4$ & $7.3\pm 0.4$ & $3.6\pm 0.3$\\
\hline $p=100$& $8.7\pm 0.2$&$8.2\pm 0.2$ & $6.8\pm 0.2$ & $3.8\pm 0.2$\\
\hline  $p=200$& $10.0\pm 0.2$ &$9.5\pm 0.3$ & $7.9\pm 0.3$ & $4.4\pm 0.2$\\
\hline  & \multicolumn{4}{c|}{ Operator distance $\|\widehat{\Omega}^{-1}-
\Sigma\|\times 10^{-3}$ }\\ 
\hline $p=30$ & $13.4\pm 4.2$ &$12.9\pm 4.0$ & $14.1\pm 4.4$  & $12.9\pm 4.0$\\
\hline $p=100$& $1.5\pm 0.4$ &$1.4\pm 0.4$  &$1.3\pm 0.4$  &$1.4\pm 0.4$ 
\\\hline 
 $p=200$ & $1.8\pm 0.2$ &$1.3\pm 0.2$ & $1.3\pm 0.2$  & $1.3\pm 0.2$  \\
\hline 
 \end{tabular}
\end{center}

\end{Table} 

~\\ {\it Comments of Tables \ref{tab1} and \ref{tab2}}: In the first scheme
$\Omega_1$, the three methods based on Cholesky decomposition exhibit a Kullback
risk close to the oracle. The
ratio of their Kullback risks over the oracle risk remains smaller than $1.4$.
The risk of the nested Lasso and the banding method is about $15\%$ smaller than
the risk of ChoSelect. We observe the same pattern for the operator distance
between precision matrices. In contrast, all these estimators have more or less
the same risks for the operator distance between the covariance matrices. The
estimator of Ledoit and Wolf is a regularized version of the empirical
covariance matrix. Its performances with respect to the Kullback loss are poor
but it behaves well with respect to the operator norms.

In the second scheme, the method of Ledoit and Wolf performs poorly with
respect to the Kullback loss functions and the first operator norm loss
function. ChoSelect performs two times better than the
nested Lasso
in terms of the Kullback discrepancy and the operator distance between precision
matrices. The banding method exhibits a far worse Kullback risk. As in the first
scheme, the three procedures based on Cholesky decomposition perform similarly
in terms of the operator distance
between covariance matrices. These last risks are high for $p=30$ because the
covariance matrix is poorly conditioned in this case and its eigenvalues are
high. 

The banding method only performs well if the Cholesky matrix $T$ is well
approximated by a banded matrix, which is not the case in the second scheme. The
nested Lasso seems to perform well when there is an exponential decay of the
coefficients as in the first scheme. However, its performance seem to be far
worse when the decay is not exponential. In contrast, ChoSelect seems to always
perform quite well. This observation corroborates the theory: indeed, we have
stated in Corollary \ref{cor_oracle_ordonne} that ChoSelect satisfies an oracle
inequality without any assumption on $\Sigma$. Finally, there no clear
interpretation for the risk with respect to the operator norm between
covariances.
\\

\begin{Table}[hptb] 
\caption{Estimation  and $95\%$ confidence interval of the power and FDR  for
the second precision model $\Omega_2$. \label{tab3}}  
\begin{center} 
\begin{tabular}{|l|ccc|ccc|}
 \hline& \multicolumn{3}{c|}{ Power$\times 10^2$ }& \multicolumn{3}{c|}{
FDR$\times 10^2$ }\\ \hline
Method & Banding &  Nested Lasso  & ChoSelect & Banding &  Nested Lasso  &
ChoSelect  \\  
\hline $p=30$ & $69.7\pm 2.3$ & $82.4\pm 0.3$  & $99.2\pm 1.1$ & $23.0\pm 1.0$ &
$17.9\pm 0.2$ & $4.7\pm 0.1$\\
\hline $p=100$& $27.0\pm 0.1$  &$82.5\pm 0.01$  &$99.4\pm 0.2$  & $3.0\pm 0.1$ &
$25.7\pm 0.2$ & $5.0\pm 0.1$ \\\hline 
 $p=200$ & $26.2\pm 0.1$ & $82.9\pm 0.1$  & $99.6\pm 0.1$ &$3.5\pm 0.1$ &
$10.0\pm 0.2$ & $4.5\pm 0.2$ \\\hline 
 \end{tabular}
\end{center}
\end{Table}

~\\ {\it Estimating the pattern of zero}. In the second scheme, we can compare
the ability of the procedures to estimate well the pattern of non-zero
coefficients (Table \ref{tab3}). The banding method does not work well since the
Cholesky factor $T$ is not banded. ChoSelect  a higher power and a lower FDR
than the nested Lasso.

\subsection{Complete Graph selection}\label{section_simulation_complete}

\subsubsection{Simulation scheme}

{\it Simulating the data}. In the first simulation study, we consider
Gaussian random vectors whose precision
matrices based on directed graphical models. 
\begin{enumerate}
 \item First, we sample a directed graph $\overrightarrow{G}$ in the following
way. For any node $i$ in $\{2,\ldots,p\}$ and any node $j<i$, we put an edge
going from $j$ to $i$ with probability $(Esp/(i-1)\wedge 0.5)$, where Esp is a
positive parameter previously chosen. Hence, the expected number of parents for
a given node is Esp$\wedge (i-1)/2$.
\item The precision matrix $\Omega^c_1$ is then defined from
$\overrightarrow{G}$.
\begin{eqnarray*}
{\bf \Omega^c_1}: \hspace{1cm}T[i,j]& \sim & U\mathrm{nif}[-1,1]  \text{ if } j
\rightarrow i \text{ in } \overrightarrow{G}\ , \\
T[i,j]& = &1 \text{ if } i=j\ \text{ and } 
T[i,j] = 0 \text{ else. }\\
S[i,i] & \sim & U\mathrm{nif}[1,2]
\end{eqnarray*}
\end{enumerate}
In the simulations, we set  $p=30,100,200$, Esp$=1,3,5$, and $n=100$.\\

In the second simulation scheme, we consider the case where the "good" ordering
is partially known. More precisely, we first sample a precision matrix
$\Omega^c_1$ according to the first simulation scheme. Then, we sample uniformly
$10$ variables and change uniformly their place in the ordering. This results in
a new precision matrix ${\bf \Omega^c_2}$. Its Cholesky factor is generally less
sparse than the one of $\Omega^c_1$. The purpose of this scheme is to check
whether our method is robust to small changes in the ordering. For this study,
we choose $p=200$, Esp$=1,3,5$, and $n=100$.

~\\
We compute the following estimators:
\begin{itemize}
\item the procedure {\bf ChoSelect}$^\text{f}$ as described in Section
\ref{section_fast_algorith}. We take the collection $\mathcal{M}^d_{\text{co}}$
with ~\\$d=\frac{n}{2.5[2+\log(n\wedge p)]}$. The collection
$\widehat{\mathcal{M}}$ is computed according to Algorithms
\ref{algorithm_faste} and \ref{Algorithme_lasso} with $k=8$. Finally, we use the
penalty (\ref{penalite_complete_log}) with $K=1.1$.
\item the procedure {\bf ChoSelect} with collection $\mathcal{M}^7_{\text{co}}$,
the penalty (\ref{penalite_complete_log}) with $K=1.1$. Since this method is
computationally prohibitive, we only apply it for $p=30$.
\item the regularization method of {\bf Ledoit and Wolf}~\cite{ledoit}.
\item the {\bf Glasso} method~\cite{aspremont}. It is computed using the
\emph{Glasso} R-package by Friedman et al.~\cite{friedman08}, while the
tuning parameter is chosen via $5$-fold cross validation based on the
likelihood. Following Rothman et al.~\cite{rothman07} and Yuan and
Lin~\cite{yuan07}, we do not penalize the diagonal of $\Omega$.
\item the {\bf Lasso} method of Huang et al.~\cite{huang06}. The
regularization parameter is calculated by $5$-fold cross validation based on the
likelihood.
\end{itemize}

For each  estimator and simulation scheme, we evaluate the Kullback loss
$\mathcal{K}(\Omega;\widehat{\Omega})$, the operator 
$\|\widehat{\Omega}-\Omega\|$, and the operator distance between the inverses
$\|\widehat{\Omega}^{-1}-\Sigma\|$. We also consider the pattern of zero
estimated by our procedure ChoSelect$^{\text{f}}$ and the Lasso of Huang
et al.~\cite{huang06}. More precisely, we evaluate the power (i.e. the
fraction of non-zero terms in $T$ estimated as non-zero) and the FDR (i.e. the
ratio of the false discoveries over the true discoveries) in the first
simulation study.  Empirical $95\%$ confidence intervals of the estimates are
also computed. The experiments are repeated $N=100$ times: $N_1=10$ precision
matrices are sampled and $N_2=10$ experiments are made for each precision matrix
sampled.

\subsubsection{Results}

\begin{Table}[hptb] 
\caption{Comparison between  ChoSelect and ChoSelect$^{\text{f}}$ using the
first covariance model 
$\Omega_1^c$ and $p=30$. \label{tab4-pfaible}}  
\begin{center}
\begin{tabular}{|l||cc|}
\hline &\multicolumn{2}{c|}{ Kullback discrepancy
$\mathcal{K}(\Omega;\widehat{\Omega})$ } \\ 
\hline Method   & ChoSelect$^{\text{f}}$ & ChoSelect   \\ \hline
\hline   Esp=1 & $0.69\pm 0.04$ & $0.69\pm 0.04$  \\
\hline  Esp=3 & $1.29\pm 0.04$ & $1.31\pm 0.05$ \\
\hline  Esp=5 & $1.95\pm 0.06$ & $1.82\pm 0.06$ \\
\hline
\end{tabular}
\end{center}

\end{Table}

~\\ {\it Comparison of  ChoSelect and ChoSelect$^{\text{f}}$}. In  Table
\ref{tab4-pfaible}, we have set $p=30$ in order to compute the method  ChoSelect
and compare it with ChoSelect$^{\text{f}}$. It seems that both methods perform
more or less similarly. When the sparsity of the Cholesky factor decreases
(Esp=5), ChoSelect$^{\text{f}}$ exhibits a slightly smaller Kullback risk. 

These simulations confirm that ChoSelect$^{\text{f}}$ exhibits similar
performances to ChoSelect with a much small computational complexity. In the
other simulations, we only compute ChoSelect$^{\text{f}}$.

\begin{Table}[hptb] 
\caption{Comparison between the procedures for the first covariance model
$\Omega_1^c$. \label{tab4}}  
\begin{center}
\begin{tabular}{|l|l|cccc|}
\hline
\multicolumn{2}{|c}{Method} & Ledoit & Glasso &  Lasso  & ChoSelect$^{\text{f}}$
 \\ \hline
\hline \multicolumn{2}{|c|}{}&\multicolumn{4}{c|}{ Kullback discrepancy
$\mathcal{K}(\Omega;\widehat{\Omega})$ }\\ 
\hline $p=100$ & Esp=1&   $7.7\pm 0.1$ & $3.7\pm 0.1$ & $3.1\pm 0.1$ &  $2.6\pm
0.1$  \\
\hline & Esp=3 & $13.9\pm 0.2$ & $9.4\pm 0.1$ & $7.2\pm 0.1$ & $5.9\pm 0.1$ \\
\hline & Esp=5 &$16.7\pm 0.2$ & $12.6\pm 0.2$ & $10.9\pm 0.2$ & $10.1\pm 0.2$ \\
\hline  $p=200$ & Esp=1& $19.4\pm 0.2$ & $9.4\pm 0.2$& $7.4\pm 0.1$ & $5.9\pm
0.1$\\
\hline & Esp=3 & $41.0\pm 0.8$ & $21.7\pm 0.3$ & $18.1\pm 0.2$ &  $13.6\pm 0.2$
\\
\hline & Esp=5 & $54.8\pm 2.1$ & $35.2\pm 0.2$ & $28.8\pm 0.3$& $24.7\pm 0.4$  
\\
\hline 
\hline \multicolumn{2}{|c|}{}&  \multicolumn{4}{c|}{ Operator distance
$\|\widehat{\Omega}- \Omega\|$ } \\
\hline $p=100$ & Esp=1& $5.5\pm 0.2$ & $4.6\pm 0.2$& $3.8\pm 0.2$ &$3.2\pm 0.1$ 
\\
\hline & Esp=3 &$8.6\pm 0.2$ & $9.3\pm 0.2$ &$6.8\pm 0.2$  & $4.6\pm 0.1$\\
\hline & Esp=5 &$11.5\pm 0.1$ & $11.9\pm 0.2$ & $9.5\pm 0.1$ & $7.6\pm 0.3$  \\
\hline  $p=200$ & Esp=1& $6.2\pm 0.1$ & $5.7\pm 0.2$ & $4.6\pm 0.1$ & $3.8\pm
0.2$ \\
\hline & Esp=3 & $10.6\pm 0.1$ & $10.7\pm 0.2$ & $8.8\pm 0.2$ &$5.4\pm 0.1$ \\
\hline & Esp=5 & $15.0\pm 0.3$ & $15.0\pm 0.2$ & $13.0\pm 0.3$ & $8.1\pm 0.2$ 
\\
\hline 
\hline \multicolumn{2}{|c|}{}&  \multicolumn{4}{c|}{ Operator distance
$\|\widehat{\Omega}^{-1}- \Sigma\|$ } 
\\
\hline $p=100$ & Esp=1& $1.5\pm 0.1$ & $1.1\pm 0.1$ & $1.1\pm 0.1$ &   $0.8\pm
0.1$ \\
\hline & Esp=3 & $4.3\pm 0.2$ & $3.9\pm 0.2$ & $5.5\pm 0.3$ &  $3.6\pm 0.3$ \\
\hline & Esp=5 &$8.4\pm 0.5$ & $9.1\pm 0.7$ & $13.0\pm 0.7$ & $8.4\pm 0.5$  \\
\hline  $p=200$ & Esp=1& $2.4\pm 0.1$ & $1.9\pm 0.1$ & $1.7\pm 0.1$ & $1.2\pm
0.1$  \\
\hline & Esp=3 & $8.3\pm 0.5$ & $6.3\pm 0.3$ & $10.7\pm 0.6$  & $6.6\pm 0.3$ \\
\hline & Esp=5 & $16.9\pm 1.4$ & $14.7\pm 1.0$ &$30.3\pm 2.9$ & $17.6\pm 1.6$ 
\\
\hline 
 \end{tabular}
\end{center}

\end{Table}

~\\ {\it Estimation of $\Omega$}. This study corresponds to the situation where
a "good" ordering of the variables is known. In Table \ref{tab4},
ChoSelect$^{\text{f}}$ has a smaller Kullback risk than the Lasso, which is
better than the Glasso, and Ledoit and Wolf's method. This is especially true
when $p$ is large. We also observe the same results it terms of the operator
distance between the precision matrices. The results for the operator distance
between covariance matrices are more difficult to interpret. It seems that the
risk of the Lasso is high, while the Glasso and ChoSelect$^{\text{f}}$ perform
more or less similarly. Ledoit and Wolf's method gives good results when
Esp=$3,5$.

\begin{Table}[hptb] 
\caption{Estimation  and $95\%$ confidence interval of the power and FDR  for
the first covariance model $\Omega_1^c$ with $p=200$.\label{tab5}} 
\begin{center}
\begin{tabular}{|l|cc|cc|}
\hline
Method &  \multicolumn{2}{|c|}{Lasso} & 
\multicolumn{2}{|c|}{ChoSelect$^{\text{f}}$}  \\ \hline
\hline &Power$\times 10^2$ & FDR$\times 10^2$ & Power$\times 10^2$ & FDR$\times
10^2$\\ 
\hline  Esp=1& $58.0\pm 0.6$ & $79.9\pm 0.4$ & $40.6\pm 0.6$ & $5.4\pm 0.6$ \\
\hline  Esp=3 & $65.3\pm 0.6$ & $72.7\pm 0.3$ & $50.9\pm 0.5$ & $9.7\pm 0.4$\\ 
\hline  Esp=5 & $67.4\pm 0.4$ & $69.2\pm 0.2$ & $52.0\pm 0.3$ & $21.1\pm 0.7$\\
\hline 
\end{tabular}
\end{center}

\end{Table}

~\\ {\it Estimation of the graph}. In Table \ref{tab5}, we compare the ability
of the procedures to estimate the underlying directed graph. This is why we only
consider the procedures based on Cholesky decomposition: the Lasso of Huang
et al. and ChoSelect$^{\text{f}}$. The Lasso exhibits a high power but
also a high FDR (larger than 50$\%$). In contrast, ChoSelect$^{\text{f}}$ keeps
the FDR reasonably small to the price of a small loss in the power. When $p$
increases, the power of the procedures decreases. These results corroborate the
results of Proposition \ref{prop_consistance}. When the number of parents
(i.e. ESP) increases, it seems that the FDR of the ChoSelect$^{\text{f}}$
increases. We recall that if one wants a lower FDR in the graph estimator, one
should choose a larger value for $K$. In practice, taking $K=2.5$ or $K=3$
enforces the FDR to be smaller than $10\%$.

\begin{Table}[hptb] 
\caption{Comparison between the procedures for the second covariance model
$\Omega_2^c$ with $p=200$.\label{tab6}}  
\begin{center}
\begin{tabular}{|l|cccc|}
\hline
Method & Ledoit & Glasso &  Lasso  & ChoSelect$^{\text{f}}$   \\ \hline
\hline &\multicolumn{4}{c|}{ Kullback discrepancy
$\mathcal{K}(\Omega;\widehat{\Omega})$ }\\ 
\hline   Esp=1& $19.2\pm 0.2$ & $8.8\pm 0.2$& $7.5\pm 0.1$ & $6.0\pm 0.1$ \\
\hline  Esp=3 & $39.6\pm 0.7$ & $21.8\pm 0.2$ & $18.9\pm 0.2$ &  $14.7\pm 0.2$
\\
\hline  Esp=5 & $56.4\pm 1.4$ & $35.6\pm 0.3$ & $32.0\pm 0.4$& $28.9\pm 0.4$  
\\
\hline 
\hline &  \multicolumn{4}{c|}{ Operator distance $\|\widehat{\Omega}- \Omega\|$
} \\
\hline  Esp=1& $6.4\pm 0.2$ & $5.6\pm 0.1$ & $4.8\pm 0.2$ & $3.8\pm 0.1$ \\
\hline  Esp=3 & $10.5\pm 0.2$ & $10.7\pm 0.2$ & $8.6\pm 0.2$ &$5.9\pm 0.2$  \\
\hline  Esp=5 & $15.0\pm 0.1$ & $14.7\pm 0.3$ & $13.6\pm 0.2$ & $9.1\pm 0.2$  \\
\hline 
\hline &  \multicolumn{4}{c|}{ Operator distance $\|\widehat{\Omega}^{-1}-
\Sigma\|$ } \\
\hline   Esp=1& $2.4\pm 0.1$ & $1.7\pm 0.1$ & $1.8\pm 0.1$ & $1.3\pm 0.1$  \\
\hline  Esp=3 & $7.6\pm 0.4$ & $6.3\pm 0.4$ & $9.3\pm 0.5$  & $6.6\pm 0.4$ \\
\hline  Esp=5 & $20.1\pm 1.6$ & $16.3\pm 1.3$ &$35.1\pm 2.5$ & $21.5\pm 1.5$  \\
\hline 
 \end{tabular}
\end{center}

\end{Table}

~\\ {\it Effect of the ordering}. In  Table \ref{tab6}, we study here the
performances of the procedures when the ordering of the variables is slightly
modified. The
Glasso method and the regularization method of Ledoit and Wolf perform as in the
first scheme since these procedures do not depend on a particular ordering of
the variables. Lasso and ChoSelect$^{\text{f}}$ procedures provide slightly
worse results than in the first scheme, especially when the sparsity decreases.
Indeed, the effect of a bad ordering is higher when the sparsity is low.
Nevertheless, ChoSelect$^{\text{f}}$ still performs better than the other
procedures for the Kullback risk and the operator distance between precision
matrices, while the Glasso and ChoSelect$^{\text{f}}$  still perform similarly
the operator distance between covariance matrices. The respective performances
are different when the ordering is completely unknown (see the
Appendix~\cite{technical}).

~\\ {\bf Conclusion}. When the ordering is known or partially known,
ChoSelect$^{\text{f}}$ has a small risk with respect to the Kullback discrepancy
and the operator distance between precision matrices. Moreover, 
ChoSelect$^{\text{f}}$ provides a good estimation of the underlying graph. It is
difficult to interpret the results for the operator distance
between the covariance matrices. If the objective is to minimize the operator
distance $\|\widehat{\Sigma}-\Sigma\|$, it seems that a direct estimation of
$\Sigma$ should be prefered to the inversion of  an estimation of $\Omega$.

\section{Discussion}\label{section_discussion}

{\it Adaptive banding problem}. ChoSelect achieves an oracle inequality and is
adaptive to the decay in the Cholesky factor $T$. We have also derived
corresponding asymptotic results for the Frobenius loss function. This procedure
is computationally competitive with the other existing methods. Finally, we
explicitly provide  the penalty and there are therefore no calibration problems
contrary to most procedures in the literature.
In a future work, we would like to study the performances of ChoSelect with
respect to the operator norm and prove corresponding minimax bounds. Bickel and
Levina have indeed proved risk bounds for their banding procedure
\cite{bickel08}. This method is based on maximum likelihood estimators as
ChoSelect. This is why we believe that ChoSelect may also satisfy fast rates of
convergence with respect to the operator distance.

~\\{\it Complete graph estimation problem}. We have derived that ChoSelect
satisfies an oracle type inequality and we have derived the minimax rates of
estimation for sparse Cholesky factors $T$. ChoSelect
is shown to achieves minimax adaptiveness to the
unknown sparsity of Cholesky factor. As in the banded case, we provide an
explicit penalty. However, this
procedure is computationally feasible only for small $p$. In contrast, the
method ChoSelect$^{\text{f}}$ introduced in Section \ref{section_fast_algorith}
shares some advantages of the previous method with a much lower computational
  cost. In Algorithm \ref{Algorithme_lasso}, we propose two collections based
on the Lasso. In practice, there are maybe smarter ways of building the
collections $\widehat{\mathcal{M}}_i$ than using the Lasso.

\section{Proofs}\label{section_proofs}

\subsection{Some notations and probabilistic tools}

First, we introduce the prediction contrasts $l_i(.,.)$. Consider $i$ be an
integer between $2$ and $p$ and let $(t,t')$ be two row vectors in
$\mathbb{R}^{i-1}$ then the contrast $l_i(t,t')$ is defined by
\begin{eqnarray}\label{definition_l}
 l_i(t,t'):=\var\left[\sum_{j=1}^{i-1}(t[j]-t'[j])X[j]\right] \ .
\end{eqnarray}
Consider a model $m_i\in\mathcal{M}_i$. We define the random variable
$\epsilon_{m_i}$ by 
\begin{eqnarray}
X[i]=\sum_{j\in m_i}-t_{i,m_i}[j]X[j]+\epsilon_{m_i}+\epsilon_i
\hspace{1.5cm}\text{a.s.}\ . \label{definition_epsilon_m}
\end{eqnarray}

By definition of $t_{i,m_i}$ in Section \ref{section_decomposition}, the
variable $\epsilon_{m_i}$ is independent of $\epsilon$ and of $X_{m_i}$.
Besides, its variance equals $l_i(t_{i,m_i},t_i)$. If follows from the
definition of $s_{i,m_i}$ that $s_{i,m_i}=l_i(t_{i,m_i},t_i)+s_i$. The vectors
$\boldsymbol{\epsilon}$ and $\boldsymbol{\epsilon}_m$ refer to the $n$ samples
of $\epsilon$ and $\epsilon_m$. For any model $m$ and any vector $Z$ of size
$n$, $\Pi_mZ$ refers to the projection of $Z$ onto the subspace generated by
$({\bf X}_i)_{i\in m}$ whereas  $\Pi^\perp_{m}Z$ stands for $Z-\Pi_m Z$. For any
subset $m$ of $\{1,\ldots,p\}$, $\Sigma_m$ denotes the covariance matrix of the
vector $X^*_m$. Moreover, we define the row vector
$Z_{m}:=X_{m}\sqrt{\Sigma_{m}^{-1}}$ 
in order to deal with  standard Gaussian vectors. Similarly to the matrix ${\bf
X}_{m}$, the  $n\times |m|$ matrix ${\bf Z}_{m}$ stands for the $n$ observations
of $Z_{m}$.

\begin{lemma}\label{lemme_basic}
The conditional Kullback-Leibler divergence
$\mathcal{K}\left(t_i,s_i;t'_i,s_i'\right)$ decomposes as 
\begin{eqnarray}\label{definition_kullback}
\mathcal{K}\left(t_i,s_i;t'_i,s_i'\right) =
\frac{1}{2}\left[\log\frac{s_i'}{s_i}+\frac{s_i}{s_i'}-1 +
  \frac{l_i(t_i,t'_i)}{s_i'}\right] \ .
\end{eqnarray}
The estimators $\widehat{t}_{i,m_i}$ and $\widehat{s}_{i,m_i}$ are expressed as
follows
\begin{eqnarray}
{\bf X}_{<i}\widehat{t}^*_{i,m_i} & = & -{\bf X}_{m_i}({\bf X}^*_{m_i}{\bf
X}_m)^{-1}{\bf X}^*_{m_i}{\bf X}_i\ , \label{expression_estimateur_t}\\
 \widehat{s}_{i,m_i} & =& \|\Pi_{m_i}^{\perp}{\bf X}_i\|_n^2
=\|\Pi_{m_i}^{\perp}(\boldsymbol{\epsilon}_{i,m_i}+\boldsymbol{\epsilon}_{i}
)\|_n^2\ .\label{expression_s}
\end{eqnarray}
\end{lemma}
This lemma is a consequence of the definitions of $\widehat{t}_{i,m_i}$,
$\widehat{s}_{i,m_i}$, and $\mathcal{K}\left(t_i,s_i;t'_i,s_i'\right)$ in
Sections \ref{section_description} and \ref{section_decomposition}.

\subsection{Proof of Proposition \ref{prte_basic_kullback_mle}}

\begin{proof}[Proof of Proposition \ref{prte_basic_kullback_mle}]
First, we decompose the Kullback-Leibler divergence into a bias term and a
variance term using Expression (\ref{definition_kullback}).
$$
\mathbb{E}\left[2\mathcal{K}\left(t_i,s_i;\widehat{t}_{i,m_i},\widehat{s}_{i,m_i
}\right)\right] = 
\mathbb{E}\left[\log\frac{\widehat{s}_{i,m_i}}{s_i}+\frac{s_i+l_i(\widehat{t}_{i
,m_i},t_i)}{\widehat{s}_{i,m_i}}-1\right]\nonumber \ .$$
By definition, $\widehat{t}_{i,m_i}$ is the least-squares estimator of $t_i$
over the set of vectors of size $i-1$ whose support is included in $m_i$ and
$-X_{<i}t^*_{i,m_i}$ is the best predictor of $X_i$ given $X_{m_{i}}$. Hence,
the prediction error  $l_i(\widehat{t}_{i,m_i},t_{i})+s_i$ equals
$l_i(\widehat{t}_{i,m_i},t_{i,m_i})+s_{i,m_i}$ and it follows that
\begin{eqnarray}
 \lefteqn{\mathbb{E}\left[2\mathcal{K}\left(t_i,s_i;\widehat{t}_{i,m_i},\widehat
{s}_{i,m_i}\right)\right] =
2\mathcal{K}\left(t_i,s_i;t_{i,m_i},s_{i,m_i}\right)}\hspace{1cm} \nonumber\\ &
&\mbox{}\hspace{1cm} +
\mathbb{E}\left[\log\frac{\widehat{s}_{i,m_i}}{s_{i,m_i}}+
\frac{l_i\left(\widehat{t}_{i,m_i},t_{i,m_i}\right)}{\widehat{s}_{i,m_i}}
+\left(\frac{s_{i,m_i}}{\widehat{s}_{i,m_i}}-1\right)\right]\ .\label{decompo1}
\end{eqnarray}
Let us compute the expectation of these three last terms. Notice that
$n\widehat{s}_{i,m_i}/s_{i,m_i}=n\|\Pi_{m_i}^{\perp} {\bf X}_i\|_n^2/s_{i,m_i}$
follows  the distribution of a $\chi^2$ distribution with $n-|m_i|$ degrees of
freedom.
\begin{eqnarray}\label{esperance1}
 \mathbb{E}\left[\frac{s_{i,m_i}}{\widehat{s}_{i,m_i}}-1\right]=\mathbb{E}\left[
\frac{n}{\chi^2(n-|m_i|)}-1\right]=\frac{|m_i|+2}{n-|m_i|-2}\ ,
\end{eqnarray}
by Lemma  5 in \cite{baraud08}. Similarly, we compute the expectation of the
logarithm as follows:
\begin{eqnarray}
 \mathbb{E}\left[\log\frac{\widehat{s}_{i,m_i}}{s_{i,m_i}}\right]& = &
\mathbb{E}\left[\log\left(\frac{\chi^2(n-|m_i|)}{n}\right)\right]
 =  \Psi(n-|m_i|)+\log\left(\frac{n-|m_i|}{n}\right) \ , \label{esperance2}
\end{eqnarray}
by definition of the function $\Psi(.)$. The last term
$l_i(\widehat{t}_{i,m_i},t_{i,m_i})/\widehat{s}_{i,m_i}$ is slightly more
difficult to handle. Let us first decompose
$l_i(\widehat{t}_{i,m_i},t_{i,m_i})$:
\begin{eqnarray*}
 l_i(\widehat{t}_{i,m_i},t_{i,m_i})& = &
(t_{i,m_i}-\widehat{t}_{i,m_i})\Sigma_{m_i}(t_{i,m_i}-\widehat{t}_{i,m_i})^*\\ 
&=&(\boldsymbol{\epsilon}_i+\boldsymbol{\epsilon}_{i,m_i})^*{\bf X}_{m_i}({\bf
X}_{m_i}^*{\bf X}_{m_i})^{-1}\Sigma_{m_i}({\bf X}_{m_i}^*{\bf X}_{m_i})^{-1}{\bf
X}_{m_i}^*(\boldsymbol{\epsilon}_i+\boldsymbol{\epsilon}_{i,m_i})
 \ ,
\end{eqnarray*}
by Lemma \ref{lemme_basic} and definition of $\epsilon_{i,m_i}$. Observe that
$\epsilon_i+\epsilon_{i,m_i}$ is independent of $X_{m_i}$. Hence, conditionally
to ${\bf X}_{m_i}$, $l_i(\widehat{t}_{i,m_i},t_{i,m_i})$ only depends on
$\boldsymbol{\epsilon}_i+\boldsymbol{\epsilon}_{i,m_i}$ through its orthogonal
projection onto the space generated by $({\bf X}_j)_{j\in m_i}$. Meanwhile,
$\widehat{s}_{i,m_i}=\|\Pi^\perp_{m_i}(\boldsymbol{\epsilon}_i+\boldsymbol{
\epsilon}_{i,m_i})\|_n^2$ is the orthogonal projection of 
$(\boldsymbol{\epsilon}_i+\boldsymbol{\epsilon}_{i,m_i})$ along the same
subspace. Thus, $l_i(\widehat{t}_{i,m_i},t_{i,m_i})$ and $\widehat{s}_{i,m_i}$
are independent conditionally to ${\bf X}_{m_i}$. Moreover,
$\widehat{s}_{i,m_i}$ is independent of ${\bf X}_{i,m_i}$. Hence,
$l_i(\widehat{t}_{i,m_i},t_{i,m_i})$ and $\widehat{s}_{i,m_i}$ are independent.
Following the proof of Lemma 2.1 in \cite{verzelen_regression}, we observe that
$\mathbb{E}[l_i(\widehat{t}_{i,m_i},t_{i,m_i})]$ is the expectation of the trace
of an inverse Wishart $Wish^{-1}(|m_i|,n)$ times $s_{i,m_i}$. We then obtain
that
\begin{eqnarray}\label{esperance3}
 \mathbb{E}\left[\frac{l_i\left(\widehat{t}_{i,m_i},t_{i,m_i}\right)}{\widehat{s
}_{i,m_i}}\right]& = &
\mathbb{E}\left[\frac{Wish^{-1}(|m_i|,n)}{\chi^2(n-|m_i|)/n}\right] =
\frac{n|m_i|}{(n-|m_i|-1)(n-|m_i|-2)}\ ,
\end{eqnarray}
since $\mathbb{E}\left[Wish^{-1}(|m_i|,n)\right]=|m_i|/(n-|m_i|-1)$ by Von Rosen
\cite{rosen1988}. Gathering identities (\ref{esperance1}), (\ref{esperance2}),
and (\ref{esperance3}) with (\ref{decompo1}) yields the first result
(\ref{decomposition_biais_variance}). Let us now compute the function
$\Psi(.)$. 
\begin{lemma}\label{calcul_log_chi2}
 For any $d$ larger than 3, 
\begin{eqnarray*}
 -\frac{1}{d-2}\leq \Psi(d)\leq 0\  \text{ and } \ \Psi(d) = -\frac{1}{d} +
\mathcal{O}\left(\frac{1}{d^2}\right) \ .
\end{eqnarray*}
\end{lemma}
The proof is given in the technical Appendix~\cite{technical}.
Since $\log(1-d/n)$ is negative, we obtain the first upper bound on $R_{n,d}$.
For any positive number $x$, $\log(1+x)\leq x$ and consequently $\log(1-x)$ is
smaller than $-x/(1-x)$ for any $x$ such that $0<x<1$. It then follows that
$\Psi(n-d)+\log(1-d/n)\geq -(d+1)/(n-d-2)$ and $R_{n,d}\geq (d+1)/[2(n-d-2)]$.
Analogously, we obtain the expansion of $R_{n,d}$ when $d/n$ goes to $0$ thanks
to Lemma \ref{calcul_log_chi2} and the Taylor expansion of the logarithm.
\end{proof}

\subsection{Proof of the risk upper bounds}\label{section_preuve_oracle}

\subsubsection{Proof of the main theorem}
\begin{proof}[Proof of Theorem \ref{main_thrm}]

This result is based on a Kullback oracle inequality for all the estimators
$(\widetilde{t}_i,\widetilde{s}_i)$ with $1\leq i\leq p$. Consider an integer
$1\leq i\leq p$. ~\\~\\
\textit{Assumption} $(\boldsymbol{\mathbb{H}}^i_{K,\eta})$: Given $K>1$ and
$\eta>0$, the collection $\mathcal{M}$ and the number $\eta$ satisfy
\begin{eqnarray}\label{Hypothese_HK_i}
 \forall m_i\in\mathcal{M}_i\ ,\,\,\,\,\,\,
\frac{\left[1+\sqrt{2H_i(|m_i|)}\right]^2|m_i|}{n-|m_i|}\leq \eta<\eta(K)\ ,
\end{eqnarray}
where we recall that $\eta(K)$ is defined in Eq.(12) in
\cite{verzelen_regression}.\\

Obviously, Assumption $(\boldsymbol{\mathbb{H}}_{K,\eta})$ is equivalent to the
union of the assumptions $(\boldsymbol{\mathbb{H}}^i_{K,\eta})$.

\begin{prte}\label{oracle_kullback}
Let $K>1$ and $\eta<\eta(K)$. Assume that $n\geq n_0(K)$, that 
$(\boldsymbol{\mathbb{H}}^i_{K,\eta})$ holds, and that the penalty function is
lower bounded as follows  
\begin{eqnarray}\label{condi1_penalite}
pen_i(m)\geq K\frac{|m|}{n-|m|}\left(1+\sqrt{2H_i(|m|)}\right)^2
\, \, \, \text{for any $m\in \mathcal{M}_i$ and some $K>1$}\ .
\end{eqnarray}
Then, the penalized estimator $(\widetilde{t}_i,\widetilde{s}_i)$ satisfies
\begin{eqnarray*}
\mathbb{E}\left[\mathcal{K}\left(t_i,s_i;\widetilde{t}_i,\widetilde{s}
_i\right)\right]\leq
L_{K,\eta}\inf_{m_i\in\mathcal{M}_i}\left[\mathbb{E}\left[\mathcal{K}\left(t_i,
s_i;\widehat{t}_{i,m},\widehat{s}_{i,m}\right)\right]+pen_i(m)\right]
+\tau_n\left[t_i,s_i,K,\eta\right]\ .
\end{eqnarray*}
The remaining term $\tau_n(t_i,s_i,K,\eta)$ is defined by
\begin{eqnarray*}
\tau_n\left[t_i,s_i,K,\eta\right]:=
 \frac{L_K}{n} + L'(K,\eta)n^{5/2}\left[1+ \mathcal{K}\left(t_i,s_i;
0,1\right)\right]\exp\left[-nL_{K,\eta}\right]\ , 
\end{eqnarray*}
where $0$ stands here for the null vector of size $i-1$.
\end{prte}

Let us apply this property for any $i$ between $1$ and $p$. Then, we get an
upper bound for $\mathbb{E}[\mathcal{K}(\Omega;\widetilde{\Omega})]$ by applying
the chain rule as in Section \ref{section_decomposition}. The risk bound
(\ref{inegalite_oracle}) follows.
\end{proof}

\begin{proof}[Proof of Proposition \ref{oracle_kullback}]

The proof of this result is mainly inspired by ideas introduced in the proofs of
Th.3 in \cite{baraud08} and of Th.3.4 in \cite{verzelen_regression}. The case
$i=1$ is a consequence of Proposition \ref{prte_basic_kullback_mle} since
$|\mathcal{M}_1|=1$. Let us assume that $i$ is larger than one. For the sake of
clarity, we forget the subscripts $i$ in the remainder of the proof.\\

Let us introduce some new notations. First, $\langle .,.\rangle_n$ is the inner
product in $\mathbb{R}^n$ associated to the norm $\|.\|_n$. Let $m$ be any model
in the collection $\mathcal{M}$.\\

We shall use the constants $\kappa_1$, $\kappa_2$, and $\nu(K)$ as defined in
the proof of Th.3.4 in \cite{verzelen_regression}. We provide their expression
for completeness although they are not really of interest.
\begin{eqnarray*}
\kappa_1  & := & \frac{\sqrt{\frac{3}{K+2}}}{1-\sqrt{\eta}-\nu(K)}\ , \, \, \,
\, \, \, \, \, \, \, \, \,
\kappa_2  := 
\frac{(K-1)\left[1-\sqrt{\eta}\right]^2\left[1-\sqrt{\eta}-\nu(K)\right]^2}{16}
\wedge 1\ ,\\
\nu(K) &:= &\left(\frac{3}{K+2}\right)^{1/6}\wedge
\frac{1-\left(\frac{3}{K+2}\right)^{1/6}}{2}\ .
\end{eqnarray*}
Besides, we introduce the positive constant $\kappa_0$ as the largest number
that satisfies
$$\kappa_0\leq 1-\frac{2}{K+1}\text{ and } \frac{K+2}{3}\leq
(1-\kappa_0)\frac{K+1.5}{2.5} \ .$$
For clarity, the proof is split into six lemmas.

\begin{lemma}\label{lemme_principal_kullback}
 \begin{eqnarray*}
  2(1-\kappa_0)\mathcal{K}\left[t,s;\widetilde{t},\widetilde{s}\right] & \leq &
2\mathcal{K}\left[t,s;\widehat{t}_m,\widehat{s}_m\right] + (1-\kappa_0)pen(m)
+\frac{l(\widetilde{t},t)}{\widetilde{s}} \left[R_1(\widehat{m})\vee
(1-\kappa_2)(1-\kappa_0)\right] \\& +& R_2(m) +
\frac{s}{\widetilde{s}}R_3(\widehat{m})+ R_4(m,\widehat{m})\ ,
 \end{eqnarray*}
where for all model $m'\in \mathcal{M}$,
\begin{eqnarray*}
R_1(m') & := & \kappa_1+ 1-\kappa_0 -
  \frac{\|\Pi^{\perp}_{m'}\boldsymbol{\epsilon}_{{m'}}\|^2_n}{l(t_{m'},t)}
  + \kappa_2(1-\kappa_0)  \varphi_{\text{max}}\left[ n({\bf Z}^*_{{m'}}{\bf
Z}_{{m'}})^{-1}\right]
   \frac{\|\Pi_{m'} (\boldsymbol{\epsilon}
  +\boldsymbol{\epsilon}_{{m'}})\|^2_n}{l(t_{m'},t)+s}\ ,\\
  & - &K(1-\kappa_0)\left[1+\sqrt{2H(|m'|)}\right]^2\frac{|m'|}{n-|m'|}  
\frac{\|\Pi_{{m'}}^\perp (\boldsymbol{\epsilon}
  +\boldsymbol{\epsilon}_{{m'}})\|^2_n}{l(t_{m'},t)+s}\ ,\\
R_2(m) &:= & 
2\frac{\langle\Pi^\perp_{m}\boldsymbol{\epsilon},\Pi^\perp_{m}\boldsymbol{
\epsilon}_{m}\rangle_n}{\widehat{s}_{m}} +
\frac{\|\Pi^\perp_{m}\boldsymbol{\epsilon}_{m}\|_n^2-l(t_{m},t)}{\widehat{s}_{m}
}\ , \\
 R_{3}({m'}) & := &
\kappa_1^{-1}\frac{\langle\Pi^\perp_{{m'}}\boldsymbol{\epsilon},
\Pi^\perp_{{m'}}\boldsymbol{\epsilon}_{{m'}}\rangle_n^2}{sl(t_{m'},t)} + 
\kappa_2(1-\kappa_0)\varphi_{\text{max}}\left[ n({\bf Z}^*_{{m'}}{\bf
Z}_{{m'}})^{-1}\right]
   \frac{\|\Pi_{m'} (\boldsymbol{\epsilon}
  +\boldsymbol{\epsilon}_{{m'}})\|^2_n}{l(t_{m'},t)+s}\\ &+&
\frac{\|\Pi_{{m'}}\boldsymbol{\epsilon}\|_n^2}{s}
- K(1-\kappa_0)\left[1+\sqrt{2H(|m'|)}\right]^2\frac{|m'|}{n-|m'|}  
\frac{\|\Pi_{{m'}}^\perp
(\boldsymbol{\epsilon}+\boldsymbol{\epsilon}_{{m'}})\|^2_n}{l(t_{m'},t)+s}\ ,\\
R_{4}(m,{m'}) & := & \left(\|\boldsymbol{\epsilon}\|_n^2-s(1-\kappa_0)\right)
\left[\frac{1}{\widehat{s}_m} -\frac{1}{\widehat{s}_{m'}}\right]   \ .
\end{eqnarray*}
\end{lemma}
This lemma gives a decomposition of the relevant terms that we have to bound.
See~\cite{technical} Sect.1.1 for a detailed computation.
In the next four lemmas, we bound each of these terms.

\begin{lemma}\label{concentration_R1}
Let us assume that $n\geq n_0(K)$, where $n_0(K)$ is defined in the proof. There
exists an event $\mathbb{B}_1$ of probability larger than $1-
L_Kn\exp\left[-nL'(K,\eta)\right]$ with $L'(K,\eta)>0$ such that
$$ R_1(\widehat{m})\mathbf{1}_{\mathbb{B}_1}\leq v(K,\eta)(1-\kappa_0)\ ,$$
where $v(K,\eta)$ is a positive constant (strictly) smaller than $1$.
\end{lemma}

\begin{lemma}\label{concentration_R3}
Assume that $n\geq n_0(K)$. Then, under the event $\mathbb{B}_1$ defined in the
proof of Lemma \ref{concentration_R1}, 
\begin{eqnarray*}
\mathbb{E}\left[\frac{s}{\widetilde{s}}R_3(\widehat{m})\mathbf{1}_{\mathbb{B}_1}
\right] \leq \frac{L_{K,\eta}}{n}\ .
\end{eqnarray*}
\end{lemma}
These two upper bounds are at the heart of the proof. The sketch of their proofs
is analogous to Lemmas 7.10 and 7.11 in \cite{verzelen_regression}. The main
tools are deviation inequalities of $\chi^2$ random variables and of the largest
eigenvalue of a Wishart matrix. See~\cite{technical} Sect.1.2 and 1.3 for
detailed proofs.\\

Since $l(\widetilde{t},t)/\widetilde{s}$ is smaller than
$2\mathcal{K}\left[t,s;\widetilde{t},\widetilde{s}\right]$, it follows that
\begin{eqnarray*}
2\mathbb{E}\left[\mathcal{K}\left(t,s;\widetilde{t},\widetilde{s}\right)\mathbf{
1}_{\mathbb{B}_1}\right] \leq
L_{K,\eta}\left\{2\mathbb{E}\left[\mathcal{K}\left(t,s;\widehat{t}_m,\widehat{s}
_m\right)\right] + pen(m)+ \mathbb{E}\left[(R_2(m)+
R_4(m,\widehat{m}))\mathbf{1}_{\mathbb{B}_1}\right]\right\}\ .
\end{eqnarray*}

\begin{lemma}\label{concentration_R2}
Assume that $n\geq n_0(K)$. Considering the event $\mathbb{B}_1$ defined in
Lemma \ref{concentration_R1}, we bound $R_2(m)$ by 
\begin{eqnarray*} 
\mathbb{E}\left[R_2(m)\mathbf{1}_{\mathbb{B}_1}\right] \leq \frac{
L_{K,\eta}}{n}\ .
\end{eqnarray*}
\end{lemma}
 See~\cite{technical} Sect.1.4 for a detailed proof.

\begin{lemma}\label{concentration_R4}
Assume that $n\geq n_0(K)$.
Considering the event $\mathbb{B}_1$ defined in Lemma \ref{concentration_R1}, we
bound $R_4(m)$ by
\begin{eqnarray*}
 \mathbb{E}\left[R_4(m,\widehat{m})\mathbf{1}_{\mathbb{B}_1}\right] \leq  L
pen(m) + n\exp\left[-nL_K\right]\ .
\end{eqnarray*}
\end{lemma}
The proofs of this lemma relies on the same ideas as the proofs of Lemma 3 in
\cite{baraud08}. See~\cite{technical} Sect.1.5 for a detailed proof.\\

Gathering these two lemmas, we control the Kullback risk of
$(\widetilde{t},\widetilde{s})$ on the event $\mathbb{B}_1$
\begin{eqnarray}
 2\mathbb{E}\left[\mathcal{K}\left(t,s;\widetilde{t},\widetilde{s}\right)\mathbf
{1}_{\mathbb{B}_1}\right] & \leq &L_{K,\eta}\left\{
2\mathbb{E}\left[\mathcal{K}\left(t,s;\widehat{t}_m,\widehat{s}_m\right)\right]
+  pen(m)\right\}
\nonumber\\ & + &\frac{L_K}{n} + (n+L)\exp\left[-nL_K\right]\ .
\label{controle_kullback1}
\end{eqnarray}
To conclude, we need to control the Kullback risk of the estimator
$(\widetilde{t},\widetilde{s})$ on the event $\mathbb{B}_1^c$.
\begin{lemma}\label{controle_risque_kullback_A}
Outside the event $\mathbb{B}_1$, the Kullback risk is upper bounded as follows:
$$\mathbb{E}\left[\mathcal{K}\left(t,s;\widetilde{t},\widetilde{s}
\right)\mathbf{1}_{\mathbb{B}_1^c}\right] \leq
L_{K,\eta}n^{5/2}\left[1+\mathcal{K}(t,s;0,1)\right]\exp\left[-nL_K\right]\ .$$
\end{lemma}
This lemma is based on H\"older's inequality and on an upper bound of the
moments of the parametric losses $\mathcal{K}(t,s;\widehat{t}_m,\widehat{s}_m)$.
A detailed proof is in the technical Appendix~\cite{technical} Sect.1.6.
Combining (\ref{controle_kullback1}) and Lemma \ref{controle_risque_kullback_A}
allows to conclude
\begin{eqnarray*}
 \mathbb{E}\left[\mathcal{K}\left(t,s;\widetilde{t},\widetilde{s} \right)\right]
& \leq &
L_{K,\eta}\left[\mathbb{E}\left[\mathcal{K}\left(t,s;\widehat{t}_m,\widehat{s}
_m\right)\right]+pen(m)\right] + \frac{L_K}{n}\\ &+& L_{K,\eta}n^{5/2}\left[1+
\mathcal{K}(t,s;0,1)\right]\exp\left[-nL_K\right]\ .
\end{eqnarray*}

\end{proof}

\subsubsection{Proof of the corollaries}

\begin{proof}[Proof of Corollary \ref{cor_oracle_ordonne}]
 The functions $H_i(.)$ equal $0$ for all the collections
$\mathcal{M}^d_{i,\text{ord}}$. Hence, the collections
$\mathcal{M}^d_{\text{ord}}$ satisfies $(\boldsymbol{\mathbb{H}}_{K,\eta})$. We
conclude by gathering Proposition \ref{prte_basic_kullback_mle} and Theorem
\ref{main_thrm}.
\end{proof}

\begin{proof}[Proof of Corollary \ref{cor_oracle_complet}]
First, we claim that for any $K>1$ the penalties (\ref{penalite_complete_log})
are lower bounded by penalties defined in (\ref{penalite_complexe}) with some
$K'>1$ if 
$$|m_i|/(n-|m_i|)\left\{1+\sqrt{2\left[1+\log\left((i-1)/|m_i|\right)\right]^2}
\right\}\leq \nu'(K)\ .$$
If we assume that $d[1+\log(p/d)\vee 0]\leq n\eta'(K)$, for some well chosen
function $\eta'(K)$, then $(\mathbb{H}_{K',\eta})$ is fulfilled and that the
risk bound (\ref{oracle_complet}) holds. A detailed proof is in the technical
Appendix~cite{technical} Sect.1.7.

\end{proof}

\begin{proof}[Proof of Proposition \ref{proposition_risque_fast}]
Under the event $\mathbb{A}_m$, the  model $m$ belongs to the collection
$\widehat{\mathcal{M}}_1\times \ldots\times \widehat{\mathcal{M}}_p$. Hence for
any $i$ in $1,\ldots p$,
$\log(\widehat{s}_{i,\widehat{m}^f_i})+pen(\widehat{m}^f_i)\leq
\log(\widehat{s}_{i,m_i})+pen(m_i)$. The rest of the proof is analogous to the
proof of Theorem \ref{main_thrm}.
\end{proof}

\subsection{Proofs of the minimax bounds}

The minimax bounds are based on Fano's method~\cite{yu}. Since the Kullback
discrepansy is not a distance, we cannot directly apply this method. Instead, we
use a modified version of Birg\'e's lemma~\cite{birgelemma} for covariance
estimation. In the sequel, 
we note $\|t\|_{l_2}$ the Euclidean norm of a vector $t$.

\begin{lemma}\label{lemme_principal_minoration_minimax}
Let  $A$  be a subset  of $\{1,\ldots ,p\}$. For any positive matrices $\Omega$
and $\Omega'$, we define the function $d(\Omega,\Omega')$ by	
\begin{eqnarray}
d(\Omega,\Omega'):=\sum_{i\in
A}\log\left[1+\frac{\|t_{i}-t'_i\|^2_{l_2}}{4}\right]+\sum_{i\in
A^c}\frac{s_i}{s'_i}+\log\left(\frac{s_i}{s'_i}\right)-1\
.\label{definition_distance_dn} 
\end{eqnarray}
Let $\Upsilon$ be a subset of square matrices of size $p$ which satisfies the
following assumptions:
\begin{enumerate}
 \item For all $\Omega\in \Upsilon,\, \varphi_{\text{max}}(\Omega)\leq 2\text{
and } \varphi_{\text{min}}(\Omega)\geq 1/2$.
 \item There exists $(\mathbf{s}_1,\mathbf{s}_2)\in[1;2]^2$  such that
$\forall\Omega\in\Upsilon,\ \forall 1\leq i\leq p$, $s_i\in
\{\mathbf{s}_1,\mathbf{s}_2\}$. 
\end{enumerate}
Setting  $\delta=\min_{\Omega,\Omega'\in \Upsilon,\Omega\neq
\Omega'}d(\Omega,\Omega')$, provided that 
$\max_{\Omega,\Omega'\in \Upsilon}\mathcal{K}(\mathbb{P}^{\otimes
n}_{\Omega};\mathbb{P}^{\otimes n}_{\Omega'})\leq \kappa_1 \log|\Upsilon|$, the
following lower bound holds 
\begin{eqnarray*}
\inf_{\widehat{\Omega}}\sup_{\Omega\in
\Upsilon}\mathbb{E}_{\Omega}\left[\mathcal{K}\left(\Omega;\widehat{\Omega}
\right)\right]\geq \kappa_2\delta\ .
\end{eqnarray*}
The numerical constants $\kappa_1$ and $\kappa_2$ are made explicit in the
proof.
\end{lemma}

The general setup of the proofs is to pick a maximal subset $\Upsilon$ of
matrices that are well separated with respect to $d(.,.)$ and such that their
Kullback discrepansy is not too large. The existence of these subsets is ensured
by technical combinatorial arguments. We
postpone the complete proofs to the technical appendix~\cite{technical} Sect.2.

\subsection{Proof of the Frobenius bounds}\label{section_preuve_frob}

We derive the Frobenius rates of convergence from the Kullback bounds. Indeed,
we prove in~\cite{technical} that
\begin{eqnarray}\label{majoration_risque_fr}
\|\sqrt{\Sigma}\Omega'\sqrt{\Sigma}-I_{p_n}\|_F^2=
4\left[\mathcal{K}\left(\Omega;\Omega'\right)\right]+o\left[\mathcal{K}
\left(\Omega;\Omega'\right)\right]\ ,
\end{eqnarray}
when $\mathcal{K}\left(\Omega;\Omega'\right)$ is close to $0$. Hence, one may
upper bound the Frobenius distance between $\Omega'$ and $\Omega$ in terms of
Kullback discrepancy using that 
\begin{eqnarray*}
 \|\Omega' -\Omega\|_F^2 & = &
tr\left[\sqrt{\Omega}\left(\sqrt{\Sigma}\Omega'\sqrt{\Sigma}-I_{p_n}
\right)\Omega\left(\sqrt{\Sigma}\Omega'\sqrt{\Sigma}-I_{p_n}\right)\sqrt{\Omega}
\right]\\
& \leq &
\varphi_{\text{\text{max}}}^2\left(\Omega\right)\|\sqrt{\Sigma}\Omega'\sqrt{
\Sigma}-I_{p_n}\|_F^2 \ .
\end{eqnarray*}
The complete proof of Corollaries \ref{minimax_ordonnee_Frobenius} and
\ref{risque_Frobenius_complete} are postponed to the technical
Appendix~\cite{technical} Sect.4.

 \section*{Acknowledgements}
I thank Elizaveta Levina and Adam Rothman for their help with the code of the
nested Lasso. 

\bibliographystyle{acmtrans-ims}

\bibliography{estimation}

\end{document}